\documentclass[a4paper,11pt]{article}
\usepackage[latin1]{inputenc}
\usepackage[english]{babel}
\usepackage{amsmath}
\usepackage{amsfonts}
\usepackage{amssymb}
\usepackage{epsfig}
\usepackage{amsopn}
\usepackage{amsthm}
\usepackage{color}
\usepackage{graphicx}
\usepackage{subfigure}
\usepackage{enumerate}
\newtheorem{theorem}{Theorem}[section]
\newtheorem{corollary}[theorem]{Corollary}
\newtheorem{lemma}[theorem]{Lemma}
\newtheorem{proposition}[theorem]{Proposition}

\newtheorem{remark}[theorem]{Remark}

\newtheorem*{theorem*}{Theorem}
\newtheorem*{lemma*}{Lemma}
\newtheorem*{remark*}{Remark}
\newtheorem*{definition*}{Definition}
\newtheorem*{proposition*}{Proposition}
\newtheorem*{corollary*}{Corollary}
\numberwithin{equation}{section}
\newcommand{\real}{\mathbb{R}}

\let\ced=\c         

\def\qed{\,\unskip\kern 6pt \penalty 500
\raise -2pt\hbox{\vrule \vbox to8pt{\hrule width 6pt
\vfill\hrule}\vrule}\par}
\definecolor{darkblue}{rgb}{0.05, .05, .65}
\definecolor{darkgreen}{rgb}{0.1, .65, .1}
\definecolor{darkred}{rgb}{0.8,0,0}
\newcommand{\beqn}{\begin{equation}}
\newcommand{\eeqn}{\end{equation}}
\newcommand{\bear}{\begin{eqnarray}}
\newcommand{\eear}{\end{eqnarray}}
\newcommand{\bean}{\begin{eqnarray*}}
\newcommand{\eean}{\end{eqnarray*}}
\begin{document}

\title{\huge \bf Peaking self-similar solutions and continuation after blow-up for a quasilinear Hardy-H\'enon equation}

\author{
\Large Razvan Gabriel Iagar\,\footnote{Departamento de Matem\'{a}tica
Aplicada, Ciencia e Ingenieria de los Materiales y Tecnologia
Electr\'onica, Universidad Rey Juan Carlos, M\'{o}stoles,
28933, Madrid, Spain, \textit{e-mail:} razvan.iagar@urjc.es},\\
[4pt] \Large Ariel S\'{a}nchez,\footnote{Departamento de Matem\'{a}tica
Aplicada, Ciencia e Ingenieria de los Materiales y Tecnologia
Electr\'onica, Universidad Rey Juan Carlos, M\'{o}stoles,
28933, Madrid, Spain, \textit{e-mail:} ariel.sanchez@urjc.es}\\
[4pt]
 \Large Erik Sarri\'{o}n-Pedralva,\footnote{Departamento de Matem\'{a}tica
	Aplicada, Ciencia e Ingenieria de los Materiales y Tecnologia
	Electr\'onica, Universidad Rey Juan Carlos, M\'{o}stoles,
	28933, Madrid, Spain, \textit{e-mail:} erik.sarrion@urjc.es}\\
[4pt] }
\date{}
\maketitle

\begin{abstract}
We consider the quasilinear Hardy-H\'enon equation
$$
u_t=\Delta u^m+|x|^{\sigma}u^p, \quad (x,t)\in\mathbb{R}^N\times(0,T),
$$
in the range of exponents and dimension
$$
N\geq3, \quad \sigma\in(-2,\infty), \quad m>1, \quad p>p_S(\sigma)=\frac{m(N+2\sigma+2)}{N-2}.
$$
We establish sharp ranges of exponents $(m,p,\sigma)$ guaranteeing existence of \emph{peaking solutions} in backward self-similar form 
$$
u(x,t)=(T-t)^{-\alpha}f(|x|(T-t)^{-\beta}).
$$ 
By peaking solutions we understand solutions blowing up at a finite time $T\in(0,\infty)$ but allowing for a continuation after the blow-up time. As a remarkable outcome, a number of ranges of exponents $\sigma$ and $m$ for which peaking solutions exist for any $p>p_S(\sigma)$ are deduced, contrasting to the previously considered exponent $\sigma=0$, where a Lepin-type critical exponent $p_L$ provides an upper bound for $p$.
\end{abstract}


\smallskip


\section{Introduction}

The goal of this paper is to address the question of establishing ranges of existence of self-similar solutions of \emph{peaking type} for the following quasilinear parabolic Hardy-H\'enon equation
\begin{equation}\label{eq1}
u_t=\Delta u^m+|x|^{\sigma}u^p, \quad (x,t)\in\mathbb{R}^N\times(0,T),
\end{equation}
with $N\geq3$ and in the range of exponents 
\begin{equation}\label{range.exp}
	\sigma\in(-2,\infty), \quad p>p_S(\sigma)>m>1.
\end{equation}
Here, $p_S(\sigma)$ denotes the Sobolev critical exponent 
\begin{equation}\label{pS}
	p_S(\sigma):=\frac{m(N+2\sigma+2)}{N-2}.
\end{equation}
A solution $u$ to Eq. \eqref{eq1} is said to be of \emph{peaking type} if there is $T\in(0,\infty)$ such that 

$\bullet$ $u$ blows up at time $T$, that is, $u(t)\in L^{\infty}(\mathbb{R}^N)$ for any $t\in(0,T)$, but $u$ becomes unbounded as $t\to T$.

$\bullet$ there exists a solution $\widetilde{u}$ to Eq. \eqref{eq1}, defined for $s\in(T,\infty)$, such that $\widetilde{u}(s)\in L^{\infty}(\mathbb{R}^N)$ for any $s\in(T,\infty)$ and 
$$
\lim\limits_{t\nearrow T}u(x,t)=\lim\limits_{s\searrow T}\widetilde{u}(x,s), \quad x\in\mathbb{R}^N.
$$

In other words, $\widetilde{u}$ has as initial trace at time $T$ exactly the blow-up profile of $u$ at time $T$. Thus, by peaking solutions we mean solutions that can be continued in a bounded, non-singular way immediately after the blow-up time. Such solutions have been discovered first for a semilinear reaction-diffusion equation with exponential source term in \cite{LT93}, and (to the best of our knowledge) the terminology of peaking solutions stems from the seminal paper by Galaktionov and V\'azquez \cite{GV97}, where the property of continuation after the blow up time has been investigated for solutions to the standard, homogeneous case of Eq. \eqref{eq1} (with $\sigma=0$).

Eq. \eqref{eq1} is a quasilinear reaction-diffusion equation, whose main feature is the competition between the effect of the mass-conservative diffusion term and the mass-increasing source term. It is a well-known fact that this competition might lead to the formation of singularities in finite time and this blow-up phenomenon has been already noticed and analyzed in the study of the spatially homogeneous equation
\begin{equation}\label{eq1.hom}
	u_t=\Delta u^m+u^p,
\end{equation}
see for example the monograph \cite{S4} dedicated to it in its main part. The main novelty of Eq. \eqref{eq1} is the presence of an unbounded weight $|x|^{\sigma}$ pondering the reaction effect, which can be either unbounded at infinity or singular at the origin, in dependence on the sign of $\sigma$. 

In recent years, equations such as Eq. \eqref{eq1}, where the source term has a power of $x$ as a factor, has been referred in literature as (parabolic) Hardy-H\'enon equations. This name is due, on the one hand, to the paper \cite{He73}, where the elliptic counterpart of the semilinear parabolic equation
\begin{equation}\label{eq0}
	u_t=\Delta u+|x|^{\sigma}u^p, \quad (x,t)\in\real^N\times(0,T),
\end{equation}
is considered as a model in astrophysics and analyzed numerically, and, on the other hand, to the paper \cite{BG84}, where Eq. \eqref{eq0} with $\sigma=-2$ and $p=1$ is considered and it is proved that existence or non-existence of weak solutions is strongly related to the best constant of Hardy's inequality. 

Two already ``classical" works dedicated to Eq. \eqref{eq0} (and to similar equations with more general weights) are \cite{Pi97, Pi98}, where an analysis of the maximum existence time and the life-span of solutions is performed. In the same period, the Fujita-type exponent, separating the ranges of non-existence and existence of non-trivial global solutions, has been identified for both Eqs. \eqref{eq0} and \eqref{eq1} in \cite{Qi98}. A number of works dedicated to the analysis of Eq. \eqref{eq0} have been published in the last decade, see for example \cite{BSTW17, CIT21, CIT22, CITT24} (to quote but a few) and references therein. In these works, existence, uniqueness and analytical properties of solutions to Eq. \eqref{eq0} are considered, looking in particular for optimal spaces for initial conditions or optimal regularity of the solutions. Eq. \eqref{eq1} and more general variations of it have been studied in \cite{AT05, Su02}, where conditions for finite time blow-up or global existence of a solutions are established and blow-up rates are deduced. Let us also quote here the papers \cite{SU21, SU24}, dealing with a finer characterization of finite time blow-up at infinity both for solutions to Eq. \eqref{eq0} and \eqref{eq1}, and \cite{MS21} giving a constructive example of blow-up of type II for solutions to Eq. \eqref{eq1} with $p$ sufficiently large. 

In the last years, two of the authors and their collaborators have started a long-term project of developing the mathematical theory and the large time behavior of solutions to parabolic Hardy-H\'enon equations. This study had started with a full analysis of self-similarity for such equations, depending on the ranges of $m$, $p$ and $\sigma$. Let us quote here, as most interesting, the results in \cite{ILS24, IL25, IMS23, IS21, IS22, IS25a, IS25b}, where radially symmetric self-similar solutions are characterized with respect to their form and asymptotic behavior as $|x|\to\infty$ in different ranges of parameters. The well-posedness and other qualitative properties of solutions in the range $1<p<m$ and $\sigma>0$ have been derived in \cite{ILS24b}, while a rather complete general theory of the singular case $\sigma\in(-2,0)$ in Eq. \eqref{eq1}, including well-posedness and optimal regularity, conditions for blow-up or global existence and large time behavior of global solutions in the range where no solution blows up, is established in \cite{IL26}.

\section{Main results}\label{subsec.mainresults}

As discussed in the previous paragraph, we devote our attention to self-similar solutions of \emph{peaking type} to Eq. \eqref{eq1}. It is a well-established fact that self-similar solutions play a fundamental role in understanding the dynamic properties and the large time behavior of solutions to nonlinear diffusion equations and, thus, their classification is an important first step towards a finer analysis of the behavior of general solutions near the blow-up time. In our case, we consider radially symmetric backward self-similar solutions in the form 
\begin{equation}\label{backward.SS}
	u(x,t)=(T-t)^{-\alpha}f(\xi), \quad \xi=|x|(T-t)^{-\beta}.
\end{equation}
Plugging the ansatz \eqref{backward.SS} into Eq. \eqref{eq1} and equating the different powers of $T-t$ appearing in its terms, we readily infer that the self-similar exponents are given by
\begin{equation}\label{SSexp}
	\alpha=\frac{\sigma+2}{\sigma(m-1)+2(p-1)}, \quad \beta=\frac{p-m}{\sigma(m-1)+2(p-1)},
\end{equation}
while the profiles $f$ are solutions to the strongly non-autonomous second order differential equation
\begin{equation}\label{ODE.backward}
	(f^m)''(\xi)+\frac{N-1}{\xi}(f^m)'(\xi)-\alpha f(\xi)-\beta\xi f'(\xi)+\xi^{\sigma}f^p(\xi)=0.
\end{equation}
Since our object of study are solutions of peaking type, we restrict ourselves to the range \eqref{range.exp}. Indeed, for $m\geq1$ and $\sigma=0$ it is a well-known fact that blow-up is complete when $m<p\leq p_S(0)$, see for example \cite[Theorem 27.2 and Remark 27.8]{QS} (for $m=1$) and \cite[Sections 7-8]{GV97} (for $m>1$), while a similar conclusion is easily deduced at least at the level of self-similar solutions in the form \eqref{backward.SS} for general $\sigma\in(-2,\infty)$, $m>1$ and $p\in(m,p_S(\sigma))$ from the analysis performed in \cite{IS25b} (we will make this point more precise later). Thus, we assume $p>p_S(\sigma)$, recalling that $p_S(\sigma)$ is defined in \eqref{pS}. Let us introduce at this point the \emph{Joseph-Lundgren critical exponent} defined as
\begin{equation}\label{pJL}
	p_{JL}(\sigma):=\begin{cases}m\frac{(N-2)^2-2(N+\sigma)(\sigma+2)+2(\sigma+2)\sqrt{(2N+\sigma-2)(\sigma+2)}}{(N-2)(N-4\sigma-10)}, & \mbox{if} \ \sigma<\frac{N-10}{4}, \\[1mm]
	\infty, & \mbox{otherwise}.  
	\end{cases}
\end{equation}
For $\sigma=0$, this exponent has been introduced for elliptic equations in \cite{JL73} and generalized to elliptic Hardy-H\'enon equations in \cite{GS81}. The critical exponent $p_{JL}(\sigma)$ plays an important role in the analysis performed in \cite{FT00, IMS26} in the semilinear case $m=1$. Finally, for $m>1$ but $\sigma=0$, $p_{JL}(0)$ strongly influences the analysis of self-similarity and continuation after blow-up in \cite[Sections 12-14]{GV97}. The general definition in \eqref{pJL} extends all these previous ones. Let us observe here that, for $\sigma\in(-2,0)$, $p_{JL}(\sigma)$ could be finite also in dimension $N\leq10$, in contrast to the well-known case when $\sigma=0$.

We also introduce the following unbounded stationary solution to Eq. \eqref{eq1} 
\begin{equation}\label{estacionaria}
\begin{split}
	&U_s(x)=C(\sigma)\lvert x\rvert^{-(\sigma+2)/(p-m)},\quad {\rm with} \\ &C(\sigma)=\left(\dfrac{m(\sigma+2)[(N-2)p-m(N+\sigma)]}{(p-m)^2}\right)^{1/(p-m)}.
\end{split}
\end{equation}
In this notation, our first result establishes a criterion for the existence of peaking solutions. 
\begin{theorem}\label{th.peaking}
	Let $m$, $p$ and $\sigma$ be as in \eqref{range.exp}. Then any self-similar solution in the form \eqref{backward.SS} to Eq. \eqref{eq1} whose profile satisfies the decay behavior 
	\begin{equation}\label{beh.inf}
		\lim\limits_{\xi\to\infty}\xi^{(\sigma+2)/(p-m)}f(\xi)=K, \quad 0<K<C(\sigma),
	\end{equation}
	can be continued after the blow-up time $T$ (and is thus a peaking solution). Moreover, the continuation $\widetilde{u}$ is also a self-similar solution, but in the forward form
\begin{equation}\label{SSforward}
	 	\widetilde{u}(x,t)=t^{-\alpha}g(|x|t^{-\beta}), \quad (x,t)\in\mathbb{R}^N\times(T,\infty),
\end{equation}
	 where $\alpha$ and $\beta$ are the exponents defined in \eqref{SSexp}.
\end{theorem}
Let us emphasize here that the condition $K\in(0,C(\sigma))$ in \eqref{beh.inf} is fundamental, as we shall see in the proof and has been also noticed, for $\sigma=0$, in \cite{GV97}.

Theorem \ref{th.peaking} justifies the name of peaking solutions given to the self-similar solutions satisfying \eqref{beh.inf} and provides a path towards proving their existence. We will thus focus in the rest of the paper on establishing (most likely optimal) ranges of existence of such peaking solutions. The first such result is very sharp.
\begin{theorem}\label{th.pSpJL}
	Let $m$, $p$ and $\sigma$ be as in \eqref{range.exp} and assume that $p\in(p_S(\sigma),p_{JL}(\sigma))$. Then, Eq. \eqref{eq1} admits infinitely many self-similar solutions of peaking type in the form \eqref{backward.SS}. The profiles of these solutions satisfy $f(0)>0$, $f'(0)=0$ and the tail behavior \eqref{beh.inf} as $\xi\to\infty$.
\end{theorem}
This result directly generalizes analogous previous results established in \cite{L90} for $m=1$ and $\sigma=0$, in \cite{FT00} for $m=1$ and $\sigma>0$, and in \cite[Theorem 12.1]{GV97} for $m>1$ and $\sigma=0$. Let us also mention here that, in the semilinear case $m=1$ and $\sigma=0$, it is shown in \cite{PQ20} that there are countably many peaking solutions in the corresponding range. As a noticeable difference with respect to the proofs given in the previously mentioned references, our proof is based on understanding the exponent $p_{JL}(\sigma)$ as a bifurcation parameter in a dynamical system obtained from \eqref{ODE.backward} after a suitable change of variable and exploit the behavior of the system before the bifurcation occurs to prove the infinity of different solutions. A numerical viewpoint of Theorem \ref{th.pSpJL} is given in Figure \ref{fig1} in Section \ref{sec.num}.

\medskip

For the description of the next results, we work from now on with exponents $p$ such that $p>p_{JL}(\sigma)$. Let us introduce the following critical exponents 
\begin{equation}\label{critexpsigma}
	\sigma_1(N):=\dfrac{N-10}{4},\qquad \sigma_2(N):=\dfrac{2(N-10)}{N-2},
\end{equation}
observing that $p>p_{JL}(\sigma)$ implies that $p_{JL}(\sigma)<\infty$, hence $\sigma<\sigma_1(N)$ and that 
\begin{equation*}
	\sigma_1(N)-\sigma_2(N)=\dfrac{(N-10)^2}{4(N-2)}>0.\end{equation*}

In order to introduce the following results, let us first describe the general context by some precedents. In the semilinear case $m=1$ and $\sigma=0$, Lepin proved in \cite{L89, L90} (see also a more precise result in \cite{PQ20}) that, in dimension $N\geq11$, there exists an exponent $p_L(0)=(N-4)/(N-10)$ such that there is at least one self-similar solution of peaking type if $p_{JL}(0)\leq p<p_L(0)$. The question of optimality of $p_L(0)$ has been solved in the affirmative sense by Mizoguchi in \cite{M04, M09}, as it was proved that there is no self-similar solution in backward form for $p>p_L(0)$, while the non-existence for $p=p_L(0)$ appears to be still open (see for example \cite[Section 53a]{QS}).

The existence of a Lepin-type exponent has been extended to the porous medium case $m>1$ (still with $\sigma=0$) in \cite[Sections 11-12]{GV97}, and a similar problem of optimality of this exponent has been raised in \cite{GV02}. Two of the authors gave in \cite{IS25a} a partial negative answer to this open question, proving that, in fact, for any $m>1$ and any $p>p_S(\sigma)$ there are self-similar solutions in the backward form \eqref{backward.SS}, but not of peaking type. Finally, in the recent paper \cite{IMS26}, the question of the existence or not of a Lepin-type exponent $p$ limiting the range of existence of peaking-type self-similar solutions has been addressed, and it was proved that, if $\sigma\geq\sigma_2(N)$ there are peaking solutions for any $p>p_S(\sigma)$, while for $N\geq11$ and $\sigma\in(0,\sigma_2(N))$, a candidate exponent $p_L(\sigma)$ is explicitely identified. This result proves that the presence of a weight $|x|^{\sigma}$ strongly changes the general outcome with respect to the one for $\sigma=0$.

In view of this context, it is natural to address the same question of the optimal ranges of existence of peaking solutions in the form \eqref{backward.SS} in the far more general case of Eq. \eqref{eq1}. In the next results we detail the existence cases (that is, the ranges of parameters $(m,p,\sigma)$ for which there exists at least one solution of peaking type) and we make the convention that non-existence of peaking type self-similar solutions is the expected outcome in all the complementary ranges (although we will not dive into the seemingly very difficult problem of proving this non-existence in the current paper). Of course, expressing the ranges of existence is technically more tedious, but the analysis will also bring some interesting qualitative conclusions, as we shall discuss after the precise statements.

Let us suppose that $N>10$, that is, $\sigma_2(N)>0$. We introduce the following critical exponents:
\begin{equation}\label{sigmastar}
	\sigma_{*}(N):=-\frac{2(N-10)}{N+6}\in(-2,0), \quad \sigma^*(N):=\frac{2(N-1)(N-10)}{3(5N-2)}>0.
\end{equation}	
Let us also define the following numbers
\begin{equation}\label{m12}
	m_{1,2}:=1+\frac{4\sigma(N+2)\pm(N-10)\sqrt{(\sigma+2)(2N-2+\sigma)}}{2N^2-15N\sigma-22N+6\sigma+20},
\end{equation}
where $m_2$ corresponds to the plus sign, and the candidates for Lepin-type exponents
\begin{equation}\label{pL}
		p_{L}^{\pm}(\sigma):=\frac{m\left(\sigma^2+4(N-7)\sigma-12\right)-2(\sigma+2)(N-10)\pm(\sigma+2)\sqrt{P(m)}}{2(\sigma(N-2)-2(N-10))},
\end{equation}
where
\begin{equation}\label{PdePL}
	\begin{split}
			P(m):=& 4(N-2)(N-4\sigma-10)+m(-8N^2+44N(\sigma+2)-8(10+7\sigma))\\
			&+m^2(4N^2+(\sigma+6)^2-4N(10+7\sigma)).
		\end{split}
\end{equation}
Due to the coincidence, for either $m=1$ or $\sigma=0$, with the already well-established Lepin exponents introduced in \cite{L90, GV97, IMS26}, we will denote $p_L(\sigma):=p_L^{-}(\sigma)$. With this notation, we have the following result dedicated to the range $-2<\sigma<\sigma_2(N)$.
\begin{theorem}\label{th.sigmasmall}
Assume that $m$, $p$, $\sigma$ satisfy \eqref{range.exp} and that $N>10$. Then, there exists at least one peaking self-similar solution for any $p\in[p_{JL}(\sigma),p_L(\sigma))$ and each one of the following ranges for $(m,\sigma)$: 
\begin{itemize}
	\item if $-2<\sigma<\sigma_*(N)$ and $m_1\leq m\leq m_2$,
	\item if $11\leq N\leq17$, $\sigma^*(N)<\sigma<\sigma_2(N)$ and any $m>1$,
	\item if $N\geq11$, $\sigma_*(N)\leq\sigma<\min\{\sigma^*(N),\sigma_2(N)\}$ and $1<m\leq m_2$,
\end{itemize}
where $m_1$, $m_2$ are defined in \eqref{m12} and $\sigma_*(N)$, $\sigma^*(N)$ are defined in \eqref{sigmastar}. The profiles of all these peaking type solutions satisfy the decay behavior \eqref{beh.inf} as $\xi\to\infty$.
\end{theorem}
An important qualitative novelty in this range with respect to the precedents is that the value of $m$ also plays an important role in limiting these ranges of existence, a fact that is not seen in the case $m=1$. Moreover, for $m>1$ and $\sigma=0$, as a particular case we also make precise the result of \cite[Theorem 12.2]{GV97}. Indeed, we notice that its original statement contains a small calculation oversight, since in some cases, according to Theorem \ref{th.sigmasmall}, limitations on $m$ are needed in order to ensure existence of peaking solutions. In particular, for $\sigma=0$ and $m>m_2=1+1/\sqrt{N-1}$ non-existence of peaking solutions is expected. Let us also notice that 
$$
\sigma_2(N)-\sigma^*(N)=-\frac{2(N-10)(N^2-18N+8)}{3(N-2)(5N-2)},
$$
which is positive if $11\leq N\leq 17$ and negative for $N\geq18$. This calculation explains the separation between the last two items in the statement of Theorem \ref{th.sigmasmall}. In particular, we notice that for $N\geq 18$ and $m>m_2$ non-existence of peaking solutions is expected. A numerical evidence of the outcome of Theorem \ref{th.sigmasmall} is given in Figure \ref{fig2} in Section \ref{sec.num}.

\medskip

Let us now turn our attention to the remaining range $\sigma_2(N)\leq\sigma<\sigma_1(N)$. The outcome of our analysis is in this case strongly contrasting to the one of Theorem \ref{th.sigmasmall}, in the sense that peaking solutions exist for any $p$ sufficiently large. Let us introduce the following number:
\begin{equation}\label{s12}
	\sigma_{12}(N):=14N-6-8\sqrt{3N^2-2N}.
\end{equation}
We give below a qualitative statement, refraining from entering a deep technical discussion at this point.
\begin{theorem}\label{th.sigmalarge}
Let $m$, $p$ and $\sigma$ be as in \eqref{range.exp} and assume that $\sigma_2(N)<\sigma<\sigma_1(N)$. Then there exists at least one peaking type self-similar solution to Eq. \eqref{eq1} for any $p\geq\max\{p_{JL}(\sigma),p_L^+(\sigma)\}$, where $p_L^+(\sigma)$ is defined in \eqref{pL}. More precisely:

$\bullet$ for $N\in\{11,12\}$ or $N\geq13$ and $\sigma_{12}(N)\leq\sigma<\sigma_1(N)$ there exist peaking self-similar solutions for any $p\geq p_{JL}(\sigma)$ and any $m>1$.

$\bullet$ for $N\geq13$ and $\sigma_2(N)<\sigma<\sigma_{12}(N)$, there exists $M_2>1$ (having a complicated explicit form given in \eqref{a3}) such that, if $m\in(1,M_2]$, then there exist peaking solutions again for any $p\geq p_{JL}(\sigma)$. In the contrary range $m>M_2$, there exist peaking solutions for any $p>p_L^+(\sigma)$.

The profiles of these peaking type solutions also satisfy the behavior \eqref{beh.inf}. 
\end{theorem}

Putting together Theorems \ref{th.pSpJL} and \ref{th.sigmalarge}, we have the following immediate consequence.
\begin{corollary}\label{cor.always}
Let $m$, $p$ and $\sigma$ be as in \eqref{range.exp}. Assume that one of the following properties holds true:

  $\bullet$ $3\leq N\leq 10$, $\sigma\in[\sigma_1(N),\infty)$ and $m>1$.
  
  $\bullet$ $N\in\{11,12\}$, $\sigma\in(\sigma_2(N),\infty)$ and $m>1$.

  $\bullet$ $N\geq13$, $\sigma\in[\sigma_{12}(N),\infty)$ and $m>1$ or $\sigma\in(\sigma_2(N),\sigma_{12}(N))$ and $m\in(1,M_2]$.

\noindent Then, for any $p>p_S(\sigma)$, Eq. \eqref{eq1} admits at least one peaking self-similar solution in the form \eqref{backward.SS}.
\end{corollary}
	
\noindent \textbf{Remark.} Besides the ranges given in Theorem \ref{th.sigmalarge}, we can also show that, if 
$$
\max\{\sigma_2(N),\sigma^*(N)\}<\sigma<\sigma_{12}(N),
$$
there exist peaking self-similar solutions for $p\in(p_{JL}(\sigma),p_L(\sigma))$, and the same is true for some more intervals indicated at the end of the Appendix. We note that an interesting and unexpected fact might appear: \emph{a gap} in the interval of $p$ guaranteeing the existence of peaking self-similar solutions occurs, either for $p\in(p_L(\sigma),p_L^+(\sigma))$, or for $p\in(p_{JL}(\sigma),p_L^+(\sigma))$, depending on some critical exponents. Such a gap has not been observed in the previous studies dedicated to either  the homogeneous case $\sigma=0$ or to the semilinear case $m=1$, such as \cite{GV97, IMS26}. Moreover, let us also notice that the exponent $p_L(\sigma)$ is no longer an upper bound for the existence of peaking self-similar solutions, as in the semilinear case or in Theorem \ref{th.sigmasmall}. On the contrary, its counterpart $p_L^+(\sigma)$ (which is deduced as the largest root of the same quadratic polynomial in $p$, as we shall see in the proof) is now a lower bound for the existence of peaking solutions. Numerical evidence of the previously explained gap are given in Figure \ref{fig3} in Section \ref{sec.num}, in contrast to Figure \ref{fig4}, where the range $\sigma\in(\sigma_{12}(N),\sigma_1(N))$ is considered.

Due to the very technical aspect of the proofs, we shall only prove in Sections \ref{sec.calc} and \ref{sec.proof} the general result for $p>\max\{p_{JL}(\sigma),p_L^+(\sigma)\}$ without distinguishing between the cases written in the two items of Theorem \ref{th.sigmalarge}. The precise distinction given in the two bullets of Theorem \ref{th.sigmalarge} will be given in an Appendix, in order to avoid, at a first reading, some tedious estimates and calculations. 

Let us point out here that it seems interesting, at least at a purely mathematical level, to consider $N$ as a real parameter in the differential equation \eqref{ODE.backward} instead of an integer representing a physical dimension. In this case, the critical dimension in Theorem \ref{th.sigmalarge} is actually not $N=13$, but $N=38/3$. That is, there are always peaking solutions for $p>p_{JL}(\sigma)$ and  $10<N<38/3$, and the detailed technical discussion postponed to the Appendix might go on identically from $N=38/3$ instead of $N=13$. 

\medskip

The proofs of the main results are based on a mix of techniques. On the one hand, an alternative formulation of \eqref{ODE.backward} as a three-dimensional autonomous quadratic dynamical system, obtained by a suitable transformation also successfully used by the authors in \cite{IS25a,IS25b}, is employed in order to establish trajectories between its critical points, that are mapped, after undoing the transformation, to profiles of peaking type self-similar solutions. On the other hand, a careful analysis adapting the original ideas from \cite{L89, L90} is needed in order to establish the rather complicated ranges described in the theorems. Finally, an analysis of the continuation after the blow-up time, justifying the name of peaking solutions, is performed by adapting to our case the ideas in \cite{GV97}. 

\medskip 

\noindent \textbf{Structure of the paper.} Since our main goal is the existence of peaking-type solutions, we start with the theory related to the continuation after the blow-up time for solutions satisfying \eqref{beh.inf}, showing how the condition $K<C(\sigma)$ is essential for this continuation. This more theoretical analysis is done in Section \ref{sec.peaking}, where Theorem \ref{th.peaking} is proved. We next present in Section \ref{sec.local} the detailed local analysis of the critical points of the system obtained as an alternative formulation of the differential equation \eqref{ODE.backward}. A number of preparatory results concerning the global analysis of the system are gathered in Section \ref{sec.prep}. This global analysis is followed by Section \ref{sec.pSpJL} containing the proof of Theorem \ref{th.pSpJL}. Possibly the most technical, but also the most important section of the paper is Section \ref{sec.calc}, where a detailed analysis of the comparison between different critical exponents involved in the existence of peaking solutions is given. This section contains actually the main estimates of the paper and explains how exponents such as $p_L(\sigma)$, $m_1$, $m_2$, $\sigma_*(N)$, $\sigma^*(N)$ appear. This section is followed by a general proof of the existence of peaking self-similar solutions contained in Section \ref{sec.proof}. We continue with Section \ref{sec.num} collecting some numerical evidence of our results, and the paper ends with an Appendix dedicated to the detailed proof of the two items of Theorem \ref{th.sigmalarge}.

\section{Peaking solutions and continuation after the blow-up time}\label{sec.peaking}

This section is devoted to the proof of Theorem \ref{th.peaking}, making precise how the continuation after the blow-up time occurs for the self-similar solutions in backward form \eqref{backward.SS} satisfying \eqref{beh.inf}. We shall notice that the condition $K<C(\sigma)$ in \eqref{beh.inf} is fundamental for the continuation. The analysis adapts the one performed in \cite{GV97}. In order to fix the notation, let us denote by $g(\cdot;A)$ the unique (local) solution to the Cauchy problem
\begin{equation}\label{CPF}
	\left\{\begin{array}{ll}(g^m)''(\zeta)+\frac{N-1}{\zeta}(g^m)'(\zeta)+\alpha g(\zeta)+\beta\zeta g'(\zeta)+\zeta^{\sigma}g^p(\zeta)=0,\\
	g(0)=A, \quad \lim\limits_{\xi\to0}\xi^{N-1}g'(\xi)=0.\end{array}\right.
\end{equation}
Indeed, the existence and uniqueness of solutions to \eqref{CPF} follows from \cite[Lemma 2.1]{ISS26} and the regularity analysis performed at the end of the Introduction of \cite{ISS26}. For $\sigma\geq-1$, the solutions to \eqref{CPF} are actually derivable at $\xi=0$ and satisfy $g'(0)=0$, while for $\sigma\in(-2,-1)$ the second initial condition in \eqref{CPF} is optimal, see \cite[Theorem 1.2]{IL25}. We then deduce by direct calculation that 
\begin{equation}\label{interm3}
U(x,t;A)=t^{-\alpha}g(\zeta;A), \quad \zeta=|x|t^{-\beta}
\end{equation}
is a radially symmetric solution to Eq. \eqref{eq1} in the self-similar form \eqref{SSforward}.

The proof of Theorem \ref{th.peaking} is a consequence of the following result.
\begin{proposition}\label{th.peak}
	Let $m$, $p$ and $\sigma$ be as in \eqref{range.exp} and let $K\in(0,C(\sigma))$, where $C(\sigma)$ is defined in \eqref{estacionaria}. Let 
	$$
u_0(x)=K|x|^{-(\sigma+2)/(p-m)}, \quad x\in\mathbb{R}^N\setminus\{0\}.
	$$
Then, there exists a solution $u$ to Eq. \eqref{eq1} such that 
\begin{equation}\label{propsol}
	u(t)\in L^{\infty}(\mathbb{R}^N), \quad t>0, \quad u(x,0)=u_0(x), \quad x\in\mathbb{R}^N.
\end{equation}
Moreover, the solution $u$ satisfying \eqref{propsol} is in the forward self-similar form \eqref{SSforward}.
\end{proposition}
\begin{proof} We split the proof into several steps.
	
\medskip 

\noindent \textbf{Step 1. Construction of a minimal solution.} Let us consider, for $n\geq1$ and $T>0$, the following approximating sequence of Cauchy problems: 
\begin{equation}\label{CPn}
\left\{\begin{array}{ll}
u_t=\Delta u^m+f_n(x,u), & (x,t)\in\mathbb{R}^N\times(0,T),\\
u(x,0)=u_{0,n}(x), & x\in\mathbb{R}^N,
\end{array}\right.
\end{equation}
where 
$$
f_n(x,u)=\min\{|x|^{\sigma},n^{\sigma}\}u^p, \quad u_{0,n}(x)=\min\{u_0(x),n\}.
$$
Since the weight multiplying the source term of \eqref{CPn} belongs to $L^{\infty}(\real^N)$ and the initial condition of the Cauchy problem \eqref{CPn} is bounded and continuous, we deduce from \cite[Proposition 2.8]{Su02} and the discussion after it that there exists a unique solution $u_n$ to \eqref{CPn} defined on a maximal interval $(0,T_n)$. Moreover, by definition, $u_{0,n}(x)\leq u_{0,n+1}(x)$ and $f_{n}(x,u)\leq f_{n+1}(x,u)$ for any $x\in\mathbb{R}^N$ and $u\in\mathbb{R}$. The comparison principle applied to the Cauchy problem \eqref{CPn} (see for example \cite[Proposition 2.2]{Su02}) then implies that 
\begin{equation}\label{interm1}
u_{n}(x,t)\leq u_{n+1}(x,t), \quad (x,t)\in\mathbb{R}^N\times(0,T_{n+1}).
\end{equation}
Since $K<C(\sigma)$, we deduce that $u_{0,n}(x)\leq u_0(x)\leq U_s(x)$ for any $x\in\mathbb{R}^N$, where $U_s(x)$ is the stationary solution \eqref{estacionaria}. Once more, the comparison principle gives 
\begin{equation}\label{interm2}
u_n(x,t)\leq U_s(x), \quad (x,t)\in\mathbb{R}^N\times(0,T_n), \quad n\geq1.
\end{equation}
We infer from gathering the information of \eqref{interm1} and \eqref{interm2}, together with the finiteness of $U_s$ in $\real^N\setminus\{0\}$, that there exists 
\begin{equation}\label{limsol}
u(x,t):=\lim\limits_{n\to\infty}u_n(x,t)\leq U_s(x), \quad (x,t)\in\mathbb{R}^N\setminus\{0\}\times(0,\infty).
\end{equation}
It readily follows by the monotone convergence theorem applied to the weak formulation of the Cauchy problem \eqref{CPn} that $u$ is a weak solution to Eq. \eqref{eq1} which is bounded outside the origin. Moreover, if $v$ is another weak solution to the Cauchy problem \eqref{eq1}-\eqref{propsol}, it is a supersolution to the approximating problems \eqref{CPn} for any $n\geq1$. Since the Cauchy problem \eqref{CPn} admits a comparison principle by \cite[Proposition 2.2]{Su02}, we deduce that $u_n\leq v$ in $\mathbb{R}^N\times(0,\infty)$ and thus $u\leq v$ in $\mathbb{R}^N\times(0,\infty)$. This simple argument ensures that the limit solution $u$ obtained in \eqref{limsol} is \emph{the minimal solution} to the Cauchy problem \eqref{eq1}-\eqref{propsol}.

\medskip

\noindent \textbf{Step 2. The minimal solution is self-similar}. We show next that the limit solution $u$ defined in \eqref{limsol} is in the forward self-similar form \eqref{SSforward}. To this end, we proceed as in the proof of \cite[Lemma 10.2]{GV97} by introducing the rescaling 
\begin{equation}\label{resc}
	\mathcal{T}_{\lambda}u(r,t)=\lambda^{\alpha}u(\lambda^{\beta}r,\lambda t), \quad r=|x|, \quad \lambda\in(0,\infty).
\end{equation} 
One can find by direct calculation, taking into account the explicit values of $\alpha$ and $\beta$ in \eqref{SSexp}, that both Eq. \eqref{eq1} and the initial condition $u_0(x)$ are invariant with respect to the rescaling \eqref{resc}, that is, $\mathcal{T}_{\lambda}u$ is a solution to the Cauchy problem \eqref{eq1}-\eqref{propsol}. The minimality of $u$ ensures that $\mathcal{T}_{\lambda}u\equiv u$ for any $\lambda\in(0,\infty)$; that is, 
$$
u(r,t)=\lambda^{\alpha}u(\lambda^{\beta}r,\lambda t), \quad (\lambda,r,t)\in(0,\infty)\times[0,\infty)\times(0,\infty).
$$
Letting $\lambda=1/t$, we have
$$
u(r,t)=t^{-\alpha}u(rt^{-\beta},1),
$$
which is equivalent to the self-similar form \eqref{SSforward}. 

\medskip 

\noindent \textbf{Step 3. The minimal solution is bounded.} We are only left to show that the minimal solution $u$ obtained in \eqref{limsol} is bounded also at the origin, in order to complete the proof of the continuation after blow-up. We split the remaining part of the proof into two cases.

\medskip

\noindent \textbf{Case 1.} Assume that there is $A>0$ sufficiently large such that the unique solution $g(\cdot;A)$ of the Cauchy problem \eqref{CPF} intersects the stationary solution $U_s$ defined in \eqref{estacionaria}; that is, there exists $\xi(A)\in(0,\infty)$ such that 
\begin{equation}\label{interm4}
g(\xi;A)<U_s(\xi), \ {\rm for} \ \xi\in(0,\xi(A)), \quad g(\xi(A);A)=U_s(\xi(A)).
\end{equation}
We consider then the function $U(\cdot,\cdot;A)$ defined in \eqref{interm3} and we perform a comparison on the parabolic cylinder 
$$
D_{\tau}:=\{(x,t): t>\tau, |x|\leq\xi(A)t^{\beta}\},
$$
for some $\tau>0$ to be chosen appropriately. Fixing $n\geq1$ and recalling that $u_n$ is the solution to the Cauchy problem \eqref{CPn}, we observe, on the one hand, by the construction of $u_n$, that 
\begin{equation}\label{comp1}
u_n(x,t)\leq U_s(x)=U(x,t;A), \quad t>0, \quad |x|=\xi(A)t^{\beta}.
\end{equation}
Indeed, the latter inequality follows from the self-similar structure \eqref{SSforward} shared by both $U(\cdot,\cdot;A)$ and $U_s(\cdot)$ and the choice of $\xi(A)$ in \eqref{interm4}. On the other hand, $u_n$ is bounded on compact sets (in space and time), hence there is $\tau(n)>0$ such that 
\begin{equation}\label{comp2}
u_n(x,\tau(n))\leq U(x,\tau(n);A), \quad |x|\leq\xi(A)\tau(n)^{\beta}.
\end{equation}
We deduce from \eqref{comp1}, \eqref{comp2} and the parabolic comparison principle applied as in the proof of \cite[Lemma 10.2]{GV97} that
$$
u_n(x,t)\leq U(x,t;A), \quad (x,t)\in D_{\tau(n)}.
$$
Moreover, we can choose $\tau(n)$ such that $\tau(n)\to0$ as $n\to\infty$, since with this choice the domain of $|x|$ for which \eqref{comp2} holds true shrinks to zero, and near $x=0$ the inequality is clear at $t=\tau(n)$. We can hence pass to the limit as $n\to\infty$ to derive from \eqref{limsol} that 
$$
u(x,t)\leq U(x,t;A), \quad (x,t)\in D_0:=\{(x,t): t>0, |x|\leq\xi(A)t^{\beta}\}.
$$
In particular, $u(0,t)\leq U(0,t;A)=t^{-\alpha}A$, which ensures the desired boundedness. 

\medskip 

\noindent \textbf{Case 2.} Assume that $g(\xi;A)<U_s(\xi)$ for any $\xi>0$ and $A\in(0,\infty)$. Let then $\kappa\in(0,C(\sigma))$. We prove next that the function
$$
V_{\kappa}(x)=\kappa|x|^{-(\sigma+2)/(p-m)}
$$
is a strict supersolution to Eq. \eqref{eq1}. Indeed, letting $r=|x|$, we have
$$
(V_{\kappa}^m)''(r)=\kappa^m\frac{m(\sigma+2)}{p-m}\left[\frac{m(\sigma+2)}{p-m}+1\right]
$$
and thus 
\begin{equation*}
	\begin{split}
(V_\kappa^m)''(r)+\frac{N-1}{r}(V_\kappa^m)'(r)&+r^{\sigma}V_\kappa^p(r)=\kappa^m\left\{\frac{m(\sigma+2)}{p-m}\left[\frac{m(\sigma+2)}{p-m}+2-N\right]+\kappa^{p-m}\right\}\\
&\times r^{-(m\sigma+2p)/(\sigma+2)}\\&=\left[\kappa^{p-m}-C(\sigma)^{p-m}\right]r^{-(m\sigma+2p)/(\sigma+2)}<0,
	\end{split}
\end{equation*}
which shows that $V_\kappa$ is a strict supersolution. Letting $A$ large and $\kappa<C(\sigma)$ sufficiently small, we deduce that $g(\cdot;A)$ intersects $V_\kappa$ and we continue as in Case 1 by replacing $U_s$ with $V_\kappa$ in the arguments.
\end{proof}
The proof of Theorem \ref{th.peaking} is now immediate.
\begin{proof}[Proof of Theorem \ref{th.peaking}]
Let 
$$
U(x,t)=(T-t)^{-\alpha}f(|x|(T-t)^{-\beta})
$$
be a self-similar solution with a profile $f$ satisfying \eqref{beh.inf}. Fix $x\in\mathbb{R}^N\setminus\{0\}$ and $r=|x|$. We then infer from \eqref{beh.inf} that 
\begin{equation*}
\begin{split}
\lim\limits_{t\to T}U(x,t)&=\lim\limits_{t\to T}(T-t)^{-\alpha}f(|x|(T-t)^{-\beta})\\
&=\lim\limits_{t\to T}(T-t)^{-\alpha}\left(|x|(T-t)^{-\beta}\right)^{(\sigma+2)/(p-m)}f(|x|(T-t)^{-\beta})\\&\times\left(|x|(T-t)^{-\beta}\right)^{-(\sigma+2)/(p-m)}=K|x|^{-(\sigma+2)/(p-m)}.
\end{split}
\end{equation*}
Since $K\in(0,C(\sigma))$, the previous limit shows that $U$ can be continued for $t>T$ by a regular, bounded function in view of Proposition \ref{th.peak} and the proof is complete.
\end{proof}

\medskip 

\noindent \textbf{Remark.} The self-similar solutions in forward form \eqref{SSforward} to Eq. \eqref{eq1} have been considered by the authors in the recent paper \cite{ISS26}. The outcome of Theorem \ref{th.peaking} adds up to the analysis in \cite{ISS26} by proving that, for any $K\in(0,C(\sigma))$, there exists a forward self-similar solution to Eq. \eqref{eq1} with the tail behavior \eqref{beh.inf} as $\zeta\to\infty$. 

\section{An autonomous dynamical system. Phase space analysis}\label{sec.local}

In this section we introduce an alternative formulation of the differential equation \eqref{ODE.backward} in the form of an autonomous dynamical system. The careful analysis of the trajectories of this dynamical system will provide important information on the behavior of the solutions to \eqref{ODE.backward} and is thus an essential step in the proofs of the main results.

Throughout this section, and unless explicitly stated otherwise, we assume that $p>p_{S}(\sigma)$, recalling that $p_S(\sigma)$ is defined in \eqref{pS}. We introduce of the following change of variables, successfully employed, for example, in \cite{IS25b,ISS26}:
\begin{equation}\label{PSchange}
X(\eta)=\frac{\alpha}{m}\xi^2f^{1-m}(\xi), \qquad Y(\eta)=\frac{\xi f'(\xi)}{f(\xi)}, \qquad Z(\eta)=\frac{1}{m}\xi^{\sigma+2}f^{p-m}(\xi),
\end{equation}
where $\alpha$ is defined in \eqref{SSexp} and $\eta:=\ln\,\xi$ is the new independent variable. Under this transformation, the interval $\xi\in(0,\infty)$ is mapped onto the whole real line $\eta\in\mathbb{R}$. Eq. \eqref{ODE.backward} is transformed by straightforward algebraic manipulations to the following autonomous quadratic dynamical system:
\begin{equation}\label{PSsyst}
\left\{\begin{array}{ll}\dot{X}=X(2-(m-1)Y), \\[2mm] \dot{Y}=X-(N-2)Y-Z-mY^2+\frac{p-m}{\sigma+2}XY, \\[2mm] \dot{Z}=Z(\sigma+2+(p-m)Y).\end{array}\right.
\end{equation}
We next study the trajectories of the phase-space system \eqref{PSsyst} and extract from them information on the corresponding self-similar profiles and their asymptotic behavior. Since the dynamical system \eqref{PSsyst} coincides with the one derived in \cite[Section 2]{IS25b}, several technical arguments are identical or differ only by minor adaptations. Thus, some proofs will be skipped, referring the interested reader to the aforementioned work for further details.

Since we are interested in nonnegative self-similar profiles, our analysis will be restricted to the region
\begin{equation*}
	\mathcal{R}_+:=(0,\infty)\times\mathbb{R}\times(0,\infty).
\end{equation*}
Indeed, it follows directly from \eqref{PSsyst} that the boundary planes $\{X=0\}$ and $\{Z=0\}$ are invariant under the flow. Consequently, trajectories starting in $\mathcal{R}_+$ remain in this region for all the values of $\eta$ for which they are defined. Observe that $Y$ is the only phase-space variable that may change sign. According to its definition in \eqref{PSchange}, such a sign change corresponds to a critical point of the profile $f(\xi)$. 

The main strategy throughout the paper consists in analyzing the trajectories of \eqref{PSsyst} and identifying those whose behavior at finite and infinite critical points yields, after undoing the transformation \eqref{PSchange}, self-similar profiles with the desired local and asymptotic properties.

\subsection{Finite critical points}

A straightforward analysis of system \eqref{PSsyst} shows that it admits four finite critical points,
\begin{equation*}
	\begin{aligned}
P_0=(0,0,0), \  P_1=&\left(0,-\frac{N-2}{m},0\right), \  P_2=\left(0,-\frac{\sigma+2}{p-m},\frac{(\sigma+2)((N-2)p-m(N+\sigma))}{(p-m)^2}\right),\\[2mm]
& P_3 =\left(\frac{2(\sigma+2)(Nm-N+2)}{[\sigma(m-1)+2(p-1)](m-1)},\frac{2}{m-1},0\right).
\end{aligned}
\end{equation*}
For later use, we denote
\begin{equation}\label{ZP2-XP3}
	\begin{split}
		Z(P_2)& :=\frac{(\sigma+2)((N-2)p-m(N+\sigma))}{(p-m)^2},\\[2mm]
		X(P_3)& :=\frac{2(\sigma+2)(Nm-N+2)}{[\sigma(m-1)+2(p-1)](m-1)}.
	\end{split}
\end{equation}
We next describe the local behavior of these critical points under the assumptions in \eqref{range.exp}. Observe first that $Z(P_2)>0$, since $p>p_S(\sigma)$ and
$$
p_S(\sigma)-\frac{m(N+\sigma)}{N-2}=\frac{m(\sigma+2)}{N-2}>0.
$$
The local behavior near three of these equilibria is summarized in the following technical result. Since the system \eqref{PSsyst} coincides with the one analyzed in \cite[Section~2]{IS25b}, the proof follows exactly the same arguments as those of \cite[Lemmas 2.1, 2.2 and 2.4]{IS25b} and will therefore be omitted.
\begin{lemma}\label{lem.P0P1P3}
Consider the dynamical system \eqref{PSsyst} and its critical points $P_0$, $P_1$, and $P_3$. Their local phase-space structure and the asymptotic behavior of the associated profiles as $\xi\to0$ are described as follows.
\begin{enumerate}
\item The linearization of \eqref{PSsyst} around $P_0$ possesses a two-dimensional unstable manifold and a one-dimensional stable manifold contained in the $Y$-axis. The unstable manifold consists of a one-parameter family of trajectories
\begin{equation}\label{lC}
(l_C): \ \left\{\begin{array}{ll}Y(\eta)=\frac{X(\eta)}{N}-\frac{Z(\eta)}{N+\sigma}+o\bigl(|X(\eta),Z(\eta)|\bigr), \\[3mm]
Z(\eta)\sim CX(\eta)^{(\sigma+2)/2},\end{array}\right. \quad {\rm as} \ \eta\to-\infty,
\end{equation}     
where $C\in[0,\infty]$, with the convention that $C=0$ corresponds to the trajectory contained in the invariant plane $\{Z=0\}$ and $C=\infty$ to the trajectory contained in the invariant plane $\{X=0\}$. Every trajectory $l_C$ with $C\in(0,\infty)$ enters the region $\mathcal{R}_+$ and, after reverting the change of variables \eqref{PSchange}, corresponds to profiles satisfying
$f(0)=A>0$ and, when $\sigma>-1$, also $f'(0)=0$.

\item The critical point $P_1$ is a saddle point with a two-dimensional unstable manifold contained in the invariant plane $\{Z=0\}$ and a one-dimensional stable manifold contained in the invariant plane $\{X=0\}$. Moreover, no trajectory emerging from $P_1$ enters the region $\mathcal{R}_+$.

\item The critical point $P_3$ is a saddle point or a saddle-focus with a two-dimensional stable manifold entirely contained in the invariant plane $\{Z=0\}$ and a one-dimensional unstable manifold consisting of a unique trajectory. The associated profiles to this unstable trajectory exhibit the asymptotic behavior
\begin{equation}\label{beh.P3f}
\lim\limits_{\xi\to0}\xi^{-2/(m-1)}f(\xi)= \left[\frac{m-1}{2m(NM-N+2)}\right]^{1/(m-1)}.
\end{equation}
\end{enumerate}
\end{lemma}
A closer inspection of the proof of \cite[Lemma 2.1]{IS25b} leads to the more precise local behavior of the profiles corresponding to trajectories going out of $P_0$: 
\begin{equation}\label{beh.P0f}
	f(\xi)\sim\left\{\begin{array}{ll}\left[D+\frac{\alpha(m-1)}{2mN}\xi^2\right]^{1/(m-1)}, & \sigma>0,\\[4mm]
		\left[\left(\frac{D}{p-1}\right)^{(m-1)/p-1}+\frac{(m-1)((1-D))}{2mN(p-1)}\xi^2\right]^{1/(m-1)}, & \sigma=0,\\[5mm]
		\left[D+\frac{p-m}{m(N+\sigma)(\sigma+2)}\xi^{\sigma+2}\right]^{-1/(p-m)}, & \sigma\in(-2,0),
	\end{array}\right. \quad {\rm as} \ \xi\to0,
\end{equation}
where $D>0$ is an arbitrary constant. 

When analyzing the critical point $P_2$, the Joseph--Lundgren exponent $p_{JL}(\sigma)$ emerges as the threshold governing the local phase-space structure. The proof follows, with the necessary adaptations, the arguments developed in \cite[Lemma 2.3]{IS25a} and \cite[Lemma 2.3]{IS25b}. Therefore, we only highlight the main differences and state the resulting properties.
\begin{lemma}\label{lem.P2}
The critical point $P_2$ is a saddle-focus for $p\in(p_S(\sigma),p_{JL}(\sigma))$ and a hyperbolic saddle for $p\geq p_{JL}(\sigma)$. In both cases, it possesses a two-dimensional stable manifold entirely contained in the invariant plane $\{X=0\}$ and a one-dimensional unstable manifold. The unique orbit lying in the unstable manifold corresponds to the stationary solution \eqref{estacionaria}.
\end{lemma}
\begin{proof}
The linearization of \eqref{PSsyst} around the critical point $P_2$ has the matrix
	\begin{equation*}
		M(P_2)=\begin{pmatrix}
			\frac{\sigma(m-1)+2(p-1)}{p-m} & 0 & 0\\[2mm]
			0 & -\frac{(N-2)(p-p_S(\sigma))}{p-m} & -1\\[2mm]
			0 & \frac{(\sigma+2)[p(N-2)-m(N+\sigma)]}{p-m} & 0
		\end{pmatrix}
	\end{equation*} 
where $p_S(\sigma)$ is defined in \eqref{pS}. Its eigenvalues satisfy
	\begin{equation*}
		\lambda_1=\dfrac{L}{p-m}>0,\ \ \lambda_2+\lambda_3=-\dfrac{(N-2)(p-p_S(\sigma))}{p-m},\ \ \lambda_2\lambda_3= \frac{(\sigma+2)[p(N-2)-m(N+\sigma)]}{p-m}.
	\end{equation*}
Since $p>p_S(\sigma)>m(N+\sigma)/(N-2)$, it follows that $\lambda_2+\lambda_3<0$ and $\lambda_2\lambda_3>0$. Therefore, either $\lambda_2$ and $\lambda_3$ are both negative real eigenvalues or they form a pair of complex conjugates with negative real part. The distinction between these two cases is determined by the sign of the discriminant
	\begin{equation*}\small
		\begin{split}
			F(p,m,\sigma,N) & = \left(\lambda_2+\lambda_3\right)^2-4\lambda_2\lambda_3\\[2mm]
			&  =\dfrac{(N-2)(N-4\sigma-10)p^2-2m\bigl(N^2-2N(\sigma+4)-2\sigma(\sigma+2)+4\bigr)p+m^2(N-2)^2}{(p-m)^2}.
		\end{split}
	\end{equation*}
	A straightforward computation shows that
	\begin{equation*}
		F\bigl(p_S(\sigma),m,\sigma,N\bigr)=-2(N-2)(\sigma+2)>0,\qquad F\bigl(p_{JL}(\sigma),m,\sigma,N\bigr)=0.
	\end{equation*}
	Moreover, the numerator of $F(p,m,\sigma,N)$ is a quadratic polynomial in $p$ whose leading coefficient is positive for $N\geq 11$, and whose roots are $p_{JL}(\sigma)$ and another root strictly smaller than $p_S(\sigma)$. Hence
	\begin{equation*}
		F(p,m,\sigma,N)>0
	\qquad\text{for }p>p_{JL}(\sigma),
	\end{equation*}
	while
	\begin{equation*}
		F(p,m,\sigma,N)<0
	\qquad\text{for }p_S(\sigma)<p<p_{JL}(\sigma).
	\end{equation*}
	We conclude that $P_2$ is a saddle-focus when $p_S(\sigma)<p<p_{JL}(\sigma)$ and a hyperbolic saddle when $p\geq p_{JL}(\sigma)$. 
	
	Since $M(P_2)$ has a block structure, the two-dimensional stable manifold is necessarily contained in the invariant plane $\{X=0\}$, while the unstable manifold is one-dimensional and, by \cite[Theorem 3.2.1]{GH}, it consists of a unique orbit. Let us remark that the straight line 
\begin{equation}\label{estacionariaLine}
	Y=-\dfrac{\sigma+2}{p-m},\qquad Z=Z(P_2)=\dfrac{(\sigma+2)[p(N-2)-m(N+\sigma)]}{(p-m)^2}.
\end{equation}
is a trajectory of the system \eqref{PSsyst} which corrresponds to the stationary solution \eqref{estacionaria} after an immediate undoing of the change of variable \eqref{PSchange}. The uniqueness of the orbit contained in the unstable manifold of $P_2$ then ensures that the only orbit contained in it is exactly \eqref{estacionariaLine}, concluding the proof.
\end{proof}
\begin{remark}\label{rem.estacionaria}
Since $\sigma\in(-2,\infty)$, the quantity under the square root in the definition \eqref{pJL} of the Joseph--Lundgren exponent is always positive. The only singular value occurs at $\sigma=\sigma_1(N)$, where the denominator vanishes and therefore $p_{JL}(\sigma)$ is not defined. This threshold plays an important role in our analysis. Indeed, as will become apparent in the proof of the main results given in Sections \ref{sec.pSpJL} and \ref{sec.proof}, the strategy of proof depends strongly on whether the Joseph--Lundgren exponent is present or not. For $\sigma<\sigma_1(N)$, the exponent $p_{JL}(\sigma)$ provides a natural threshold leading to a specific phase-space structure and requiring a separate analysis. On the other hand, for $\sigma\geq\sigma_1(N)$ we have $p_{JL}(\sigma)=\infty$ according to \eqref{pJL}.
Hence, in the parameter range considered throughout this paper, namely $p>p_S(\sigma)$, the critical value $\sigma=\sigma_1(N)$ separates the regimes in which the Joseph--Lundgren exponent is present in the analysis from those in which it is not.
\end{remark}

\subsection{Critical points at infinity} 

In order to analyze the critical points at infinity of system \eqref{PSsyst}, we employ the Poincar\'e compactification and study the associated flow on the Poincar\'e hypersphere. Following the general theory developed in \cite[Section 3.10]{Pe}, we introduce the projective coordinates
$$
X=\frac{\overline{X}}{W}, \qquad Y=\frac{\overline{Y}}{W}, \qquad Z=\frac{\overline{Z}}{W}.
$$
To the best of our knowledge, the first use of this compactification technique in the study of self-similar solutions to porous-medium-type equations dates back to the work of Hulshof \cite{Hu91}, where it was applied in a phase-plane setting. According to \cite[Theorem 4, Section 3.10]{Pe}, the critical points at infinity are obtained as the solutions of the reduced algebraic system
\begin{equation*}
\left\{\begin{array}{ll}\frac{1}{\sigma+2}\overline{X}\overline{Y}[(p-m)\overline{X}-(\sigma+2)\overline{Y}]=0,\\[2mm]
(p-1)\overline{X}\overline{Z}\overline{Y}=0,\\[2mm]
\frac{1}{\sigma+2}\overline{Y}\overline{Z}[p(\sigma+2)\overline{Y}-(p-m)\overline{X}]=0,\end{array}\right.
\end{equation*}
together with the requirement that they belong to the equator of the Poincar\'e hypersphere. This amounts to imposing $W=0$ and the normalization condition
$$
\overline{X}^2+\overline{Y}^2+\overline{Z}^2=1.
$$ 
Since $\overline X\geq0$ and $\overline Z\geq0$, a straightforward computation yields the following critical points at infinity:
\begin{equation*}
\begin{split}
Q_1=(1,0,0&,0), \ \ Q_{2,3}=(0,\pm1,0,0), \ \ Q_4=(0,0,1,0), \ \ Q_{\gamma}=\left(\gamma,0,\sqrt{1-\gamma^2},0\right),\\
&Q_5=\left(\frac{\sigma+2}{\sqrt{(\sigma+2)^2+(p-m)^2}},\frac{p-m}{\sqrt{(\sigma+2)^2+(p-m)^2}},0,0\right),
\end{split}
\end{equation*}
with $\gamma\in(0,1)$. We next examine the critical points at infinity identified above. Since the analysis is largely analogous to that performed in \cite{IMS26,IS25b}, we restrict ourselves to outlining the main arguments and refer the interested reader to those works for the technical details.

For the critical points $Q_1$, $Q_5$, and the family $Q_\gamma$, it is convenient to employ the local chart associated with the projection onto the $X$-direction. Following \cite[Theorem 5(a), Section 3.10]{Pe}, we introduce the variables
\begin{equation}\label{change2}
	x=\frac{1}{X}, \qquad y=\frac{Y}{X}, \qquad z=\frac{Z}{X},
\end{equation}
leading to the system
\begin{equation}\label{PSinf1}
\left\{\begin{array}{ll}\frac{dx}{d\eta_1}=x[(m-1)y-2x],\\[2mm]
\frac{dy}{d\eta_1}=-y^2+\frac{p-m}{\sigma+2}y+x-Nxy-xz,\\[2mm]
\frac{dz}{d\eta_1}=z[(p-1)y+\sigma x],\end{array}\right.
\end{equation}
where the derivatives in \eqref{PSinf1} are taken with respect to the independent variable $\eta_1$ defined implicitly as
\begin{equation}\label{changexi}
\frac{d\eta_1}{d\xi}=\frac{\alpha}{m}\xi f^{1-m}(\xi)=\dfrac{X(\xi)}{\xi}.
\end{equation}
In the coordinates introduced in \eqref{change2}, the critical points $Q_1$, $Q_5$, and the family $Q_\gamma$ become
$$
Q_1=(0,0,0), \ \ Q_5=\left(0,-\frac{p-m}{\sigma+2},0\right), \ \ Q_{\gamma}=(0,0,\kappa), \ \ {\rm where} \ \kappa=\frac{\sqrt{1-\gamma^2}}{\gamma}\in(0,\infty).
$$
The characterization of the critical points $Q_1$, $Q_\gamma$, and $Q_5$ follows the same lines as the proofs of \cite[Lemmas 3.1 and 3.2]{IS25b}. Since no substantial modifications are required, we omit the details and state the corresponding result below.
\begin{lemma}\label{lem.Q1QgQ5}
\begin{enumerate}
  \item The critical point $Q_1$ possesses a two-dimensional center manifold with stable direction of the flow and a one-dimensional unstable manifold contained in the $y$-axis. The profiles corresponding to the trajectories lying in the center manifold satisfy
  \begin{equation}\label{beh.Q1}
  	\lim_{\xi\to\infty}\xi^{(\sigma+2)/(p-m)}f(\xi)=K\in(0,\infty).
  \end{equation}
      
  \item The critical point $Q_5$ is a saddle point with a two-dimensional unstable manifold and a one-dimensional stable manifold, the latter being contained in the invariant $y$-axis of system \eqref{PSinf1}. Trajectories going out of $Q_5$ give rise to profiles displaying a dead-core behavior; that is,
  \begin{equation*}\label{beh.Q5.deadcore}
  	f(0)=0, \ \xi\in[0,\xi_0],\qquad f'(\xi)>0,\ \xi\in(\xi_0,\xi_0+\delta),\qquad (f^m)'(\xi_0)=0,
  \end{equation*}
 for some $\xi_0\in(0,\infty)$ and $\delta>0$.
   
  \item Let
  \begin{equation*}
  	\gamma=\gamma_0:=\dfrac{(p-1)(\sigma+2)}{\sqrt{[\sigma(m-1)+2(p-1)]^2+(p-1)^2(\sigma+2)^2}},
  \end{equation*}
  Then the critical point $Q_{\gamma_0}$ admits a unique trajectory entering it from the finite region of the phase space associated with the system \eqref{PSinf1}. The corresponding profiles satisfy
  \begin{equation}\label{beh.Qg}
  	f(\xi)\sim\left(\dfrac{1}{p-1}\right)^{1/(p-1)}\xi^{-\sigma/(p-1)},\ \ \text{as}\ \ \xi\to\infty,
  \end{equation}
  For every $\gamma\in(0,1)$ with $\gamma\neq\gamma_0$, all the local invariant manifolds associated with $Q_\gamma$ are contained in the plane $\{x=0\}$ and therefore do not generate trajectories arising from the finite part of the phase space. 
\end{enumerate}
\end{lemma}
We observe that, depending on the sign of $\sigma\in(-2,\infty)$, the behavior \eqref{beh.Qg} corresponds either to a tail (if $\sigma>0$), to a positive horizontal asymptote (if $\sigma=0$) or to an unbounded profile (if $\sigma<0$). Moreover, when $\sigma=0$, the orbit entering $Q_{\gamma_0}$ is entirely contained in the invariant plane $\{Y=0\}$ and corresponds to the constant profile
\begin{equation*}
	f(\xi)=\left(\frac{1}{p-1}\right)^{1/(p-1)}.
\end{equation*}

\begin{remark}\label{rem.Ck}
	Similarly as in the proof of \cite[Lemma 2.3]{IMS26}, one can show that, for every $K\in(0,\infty)$, there exists a unique trajectory entering the critical point $Q_1$ through its center manifold and satisfying \eqref{beh.Q1} with that precise value of $K$.
	We denote this trajectory by $C_K$. In this way, the family $\{C_K\}_{K>0}$ provides a natural parametrization of all trajectories entering $Q_1$ along its center manifold. In particular, the stationary solution \eqref{estacionaria}, represented in the phase space by the trajectory \eqref{estacionariaLine}, coincides with $C_{Z(P_2)}$, where $Z(P_2)$ is defined in \eqref{ZP2-XP3}. Hence, the profiles satisfying \eqref{beh.inf} correspond to trajectories $C_K$ with $K<Z(P_2)$, while trajectories with $K>Z(P_2)$ yield profiles whose decay at infinity is faster than that of the stationary solution. In view of Proposition \ref{th.peak} and Theorem \ref{th.peaking}, the peaking solutions correspond to trajectories $C_K$ entering $Q_1$ with $K<Z(P_2)$.
\end{remark}

Let us also recall here the approximation of the center manifold of $Q_1$ obtained in the proof of \cite[Lemma 3.1]{IS25b}
	\begin{equation}\label{centerQ1}
		y=-\frac{\sigma+2}{p-m}x
		-\frac{(\sigma+2)^2[p(N-2)-m(N+\sigma)]}{(p-m)^3}x^2
		+\frac{\sigma+2}{p-m}xz
		+x\,o\bigl(|(x,z)|^2\bigr).
	\end{equation}
Taking into account that, on the trajectory $C_K$ with $K\in(0,\infty)$ we have $z(\eta_1)\sim Kx(\eta_1)$ as $\eta_1\to\infty$, we find that \eqref{centerQ1} reads
$$
y=-\frac{\sigma+2}{p-m}x+\frac{\sigma+2}{p-m}x^2(K-Z(P_2))+x\,o\bigl(|(x,z)|^2\bigr),
$$
which shows from which side of the plane $\{Y=-(\sigma+2)/(p-m)\}$ the trajectory $C_K$ reaches $Q_1$. The previous expansion,	together with the asymptotic behavior \eqref{beh.Q1}, also implies, after reverting the changes of variables \eqref{change2} and \eqref{PSchange}, that the derivatives of the profiles entering $Q_1$ satisfy
	\begin{equation*}
		\lim_{\xi\to\infty}\xi^{(\sigma+2)/(p-m)+1}f'(\xi)\in(0,\infty).
	\end{equation*}

\medskip

To study the local dynamics near the critical points $Q_2$ and $Q_3$, we consider the chart associated with the projection onto the $Y$-direction. Following \cite[Theorem 5(b), Section 3.10]{Pe}, we arrive in a similar way as in the previous analysis at the system
\begin{equation}\label{PSinf2}
\left\{\begin{array}{lll}\pm\frac{dx}{d\eta_2}=-x-Nxw+\frac{p-m}{\sigma+2}x^2+x^2w-xzw,\\[2mm]
\pm\frac{dz}{d\eta_2}=-pz-(N+\sigma)zw+\frac{p-m}{\sigma+2}xz+xzw-z^2w,\\[2mm]
\pm\frac{dw}{d\eta_2}=-mw-(N-2)w^2+\frac{p-m}{\sigma+2}xw+xw^2-zw^2,\end{array}\right.
\end{equation}
where the new variables are defined by
\begin{equation*}
	x=\frac{X}{Y}, \qquad z=\frac{Z}{Y}, \qquad w=\frac{1}{Y},
\end{equation*}
and the independent variable $\eta_2$ is implicitly determined by 
\begin{equation*} 
	\dfrac{d\eta_2}{d\xi}=\dfrac{Y(\xi)}{\xi}.
\end{equation*}
In these coordinates, the flow of \eqref{PSsyst} in a neighborhood of $Q_2$ and $Q_3$ is locally topologically equivalent to the flow of \eqref{PSinf2} near the origin, with the minus sign corresponding to $Q_2$ and the plus sign to $Q_3$. The analysis of the local dynamics is straightforward and follows the same lines as the proof of \cite[Lemma 3.3]{IS25b}. We therefore omit the details and state the resulting properties below.
\begin{lemma}\label{lem.Q23}
	The critical points $Q_2$ and $Q_3$ are, respectively, an unstable node and a stable node. The trajectories emerging from $Q_2$ correspond to profiles $f$ for which there exist $\xi_0\in(0,\infty)$ and $\delta>0$ such that
	\begin{equation}
		f(\xi_0)=0,\qquad
		(f^m)'(\xi_0)=C>0,\qquad
		f(\xi)>0 \quad \text{for}\quad \xi\in(\xi_0,\xi_0+\delta).
	\end{equation}
	Similarly, the trajectories entering $Q_3$ correspond to profiles $f$ for which there exist $\xi_0\in(0,\infty)$ and $\delta\in(0,\xi_0)$ such that
	\begin{equation}
		f(\xi_0)=0,\qquad
		(f^m)'(\xi_0)=-C<0,\qquad
		f(\xi)>0 \quad \text{for}\quad \xi\in(\xi_0-\delta,\xi_0).
	\end{equation}
\end{lemma}

We are therefore left with the critical point $Q_4$, whose local analysis is considerably more delicate due to its non-hyperbolic nature. Following the approach developed in \cite[Section 4 and Appendix]{IMS25}, we have the following non-existence result, which can be derived from the arguments of \cite[Appendix]{IMS25}; see also \cite[Lemma 3.4]{IS25b} for an analogous statement in a closely related setting.
\begin{lemma}\label{lem.Q4}
	The critical point $Q_4$ is not connected to the finite part of the phase space associated with system \eqref{PSsyst}. More precisely, no orbit passing through the finite region of the phase space enters or emanates from $Q_4$.
\end{lemma}

\section{Some preparatory results of global analysis}\label{sec.prep}

Before proceeding to the proof of the main existence theorems, we establish a number of preliminary results concerning the global phase space structure of the system \eqref{PSsyst}. These properties will provide the main ingredients for the shooting arguments developed in the subsequent sections.

A key step in the analysis of the trajectories of interest in the phase space associated to the system \eqref{PSsyst} consists in understanding the global behavior of several distinguished trajectories contained in invariant subsets. These trajectories play the role of separating or bounding orbits belonging to the invariant manifolds involved in the main shooting procedures and will serve as reference configurations throughout the global analysis.

To be more specific, we study the global dynamics of some specific orbits lying in the invariant planes $\{X=0\}$ and $\{Z=0\}$ of the system \eqref{PSsyst} (and, when convenient, in the corresponding invariant subsets arising in the compactified systems introduced in Section \ref{sec.local}). We begin with the invariant plane $\{X=0\}$, where the system \eqref{PSsyst} reduces to
\begin{equation}\label{systX0}
	\left\{\begin{array}{ll}
		\dot{Y}=-(N-2)Y-Z-mY^2,\\[2mm]
		\dot{Z}=Z(\sigma+2+(p-m)Y).
	\end{array}\right.
\end{equation}
The following result gives the phase-plane configuration of the system \eqref{systX0}.
\begin{proposition}\label{prop.X0}
	Let $N\geq3$ and assume that $\sigma$, $m$ and $p$ satisfy \eqref{range.exp}. Consider the invariant plane $\{X=0\}$ associated with the system \eqref{systX0}. Then the following properties hold:
\begin{enumerate}
		\item The orbit denoted by $l_{\infty}$ in \eqref{lC}, emerging from the critical point $P_0$ inside the plane $\{X=0\}$ is a heteroclinic connection. More precisely,
		\begin{itemize}
			\item if $p>p_S(\sigma)$, then $l_{\infty}$ connects $P_0$ to $P_2$;
			\item if $p=p_S(\sigma)$, then $l_{\infty}$ connects $P_0$ to $P_1$ and is explicitly given by
			\begin{equation}\label{cylinder}
				Z=-\frac{N+\sigma}{N-2}(mY+N-2)Y.
			\end{equation}
		\end{itemize}
		
		\item Assume that $p>p_S(\sigma)$. Then the unique trajectory contained simultaneously in the stable manifold of $P_1$ and in the invariant plane $\{X=0\}$ is a heteroclinic orbit connecting the critical point at infinity $Q_2$ to $P_1$.
	\end{enumerate}
\end{proposition}

The proof of Proposition \ref{prop.X0} is totally identical to the one of either \cite[Lemmas 5.1--5.2]{IMS25} or \cite[Lemma 2.8]{IMS26}, to which we refer. For the reader's convenience, we give in Figure \ref{figX0} a visual representation of the results stated in Proposition \ref{prop.X0}, where the arrows indicate the direction of the flow of the system \eqref{systX0} on or across the relevant curves for the proof. In particular, the curve directly connecting $P_0$ to $P_1$ in Figure \ref{figX0} is the explicit curve \eqref{cylinder}, which is fundamental in the proof since the direction of the flow across it is always the same, while the curve connecting $P_0$ to $P_1$ but passing through $P_2$ in Figure \ref{figX0} is the isocline $\dot{Y}=0$.

\begin{figure}[ht!]
	\begin{center}
		\includegraphics[width=8.5cm,height=6cm]{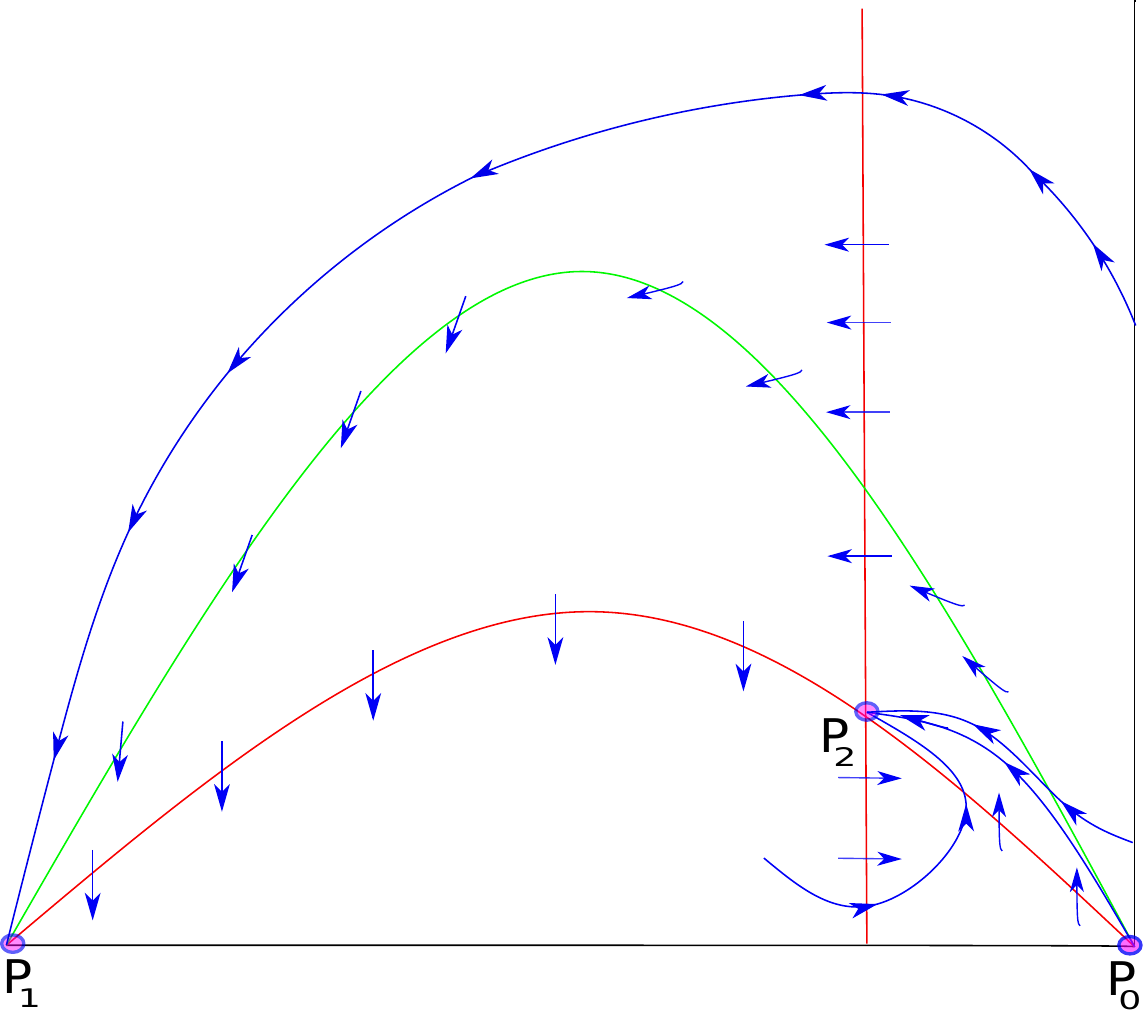}
	\end{center}
	\caption{Trajectories in the invariant plane $\{X=0\}$ for $p>p_S(\sigma)$.}\label{figX0}
\end{figure}

We next investigate the distinguished trajectory corresponding to the value $C=0$ in \eqref{lC}, which is entirely contained in the invariant plane $\{Z=0\}$. Restricting \eqref{PSsyst} to this plane, we obtain the reduced system
\begin{equation}\label{systZ0}
	\left\{\begin{array}{ll}
		\dot{X}=X(2-(m-1)Y),\\[2mm]
		\dot{Y}=X-(N-2)Y-mY^2+\dfrac{p-m}{\sigma+2}XY.
	\end{array}\right.
\end{equation}
Recalling the change of variables \eqref{change2}, we note that the invariant plane $\{Z=0\}$ corresponds to the plane $\{z=0\}$ in the compactified system \eqref{PSinf1}, which reduces to 
\begin{equation}\label{systz0}
	\left\{\begin{array}{ll}
		\dot{x}=x[(m-1)y-2x],\\[2mm]
		\dot{y}=-y^2-\dfrac{p-m}{\sigma+2}y+x-Nxy.
	\end{array}\right.
\end{equation}
The configuration of the phase plane associated to the system \eqref{systZ0} (or to \eqref{systz0}) is established in the following proposition.

\begin{proposition}\label{prop.Z0}
	Let $N\geq3$, and assume that $\sigma$, $m$ and $p$ satisfy \eqref{range.exp}. Then the following properties hold.
	\begin{enumerate}
		\item The orbit corresponding to $C=0$ in \eqref{lC}, denoted by $l_0$, is a heteroclinic connection from $P_0$ to $P_3$.
		
		\item The unique orbit contained in the unstable manifold of $Q_5$ within the plane $\{Z=0\}$ is a heteroclinic connection from $Q_5$ to $P_3$.
		
		\item The unique orbit entering $Q_1$ through its center manifold within the plane $\{Z=0\}$ is a heteroclinic connection from $P_1$ to $Q_1$.
	\end{enumerate}
\end{proposition}
\begin{proof}
	The argument follows the same general strategy as in \cite[Lemma 4.2]{IS25b} and \cite[Lemma 2.9]{IMS26}, although we provide the details since this phase-plane analysis will play an important role in the subsequent global study. We divide the proof into several steps. 
	
	\medskip
	
	\noindent \textbf{Step 1. Positive invariance of the first quadrant.} We first describe the critical points governing the dynamics on the invariant plane $\{Z=0\}$. In the reduced system \eqref{systZ0}, the critical points are
	\begin{equation*}
		P_0=(0,0),\qquad
		P_1=\left(0,-\dfrac{N-2}{m}\right),\qquad
		P_3=\left(\dfrac{2(\sigma+2)(Nm-N+2)}{[\sigma(m-1)+2(p-1)](m-1)},\dfrac{2}{m-1}\right),
	\end{equation*}
	while the compactified system \eqref{systz0} contains the critical points at infinity
	\begin{equation*}
		Q_1=(0,0),\qquad
		Q_5=\left(0,\dfrac{p-m}{\sigma+2}\right).
	\end{equation*}
	By Lemma \ref{lem.P0P1P3}, the critical point $P_0$ is a saddle point for system \eqref{systZ0}, while $P_1$ is an unstable node and $P_3$ is a stable node. Likewise, Lemma \ref{lem.Q1QgQ5} shows that $Q_1$ possesses a one-dimensional center-stable manifold, whereas $Q_5$ is a saddle point for the compactified system \eqref{systz0}. By the Stable Manifold Theorem \cite[Section 2.7]{Pe}, the unique orbit contained in the unstable manifold of $P_0$ leaves this critical point tangent to the eigenvector $e_1=(N,1)$, and therefore enters the first quadrant $\{X>0,\ Y>0\}$. Moreover, the direction of the flow of the system \eqref{systZ0} across the line $\{Y=0\}$ points towards the region $\{Y>0\}$. Hence, the first quadrant is positively invariant and the orbit emerging from $P_0$ remains in the region $\{X>0,\ Y>0\}$ for all future times. 
	
	\medskip
	
\noindent \textbf{Step 2. No periodic orbits in the first quadrant.} To determine its $\omega$-limit set, we next prove that system \eqref{systZ0} does not possess limit cycles in the first quadrant. To this end, we apply Dulac's criterion \cite[Theorem 2, Section 3.9]{Pe}. More precisely, we multiply the vector field of \eqref{systZ0} by a factor $X^aY^b$, where
	\begin{equation*}
		a=-\dfrac{2mN-N-2m+2}{mN-N+2},
		\qquad
		b=-\dfrac{mN-N+2m+2}{mN-N+2}.
	\end{equation*}
	A direct computation shows that the divergence of this weighted vector field is given by
	\begin{equation*}
			D(X,Y)
			:=\frac{\partial}{\partial X}\bigl(X^aY^bX'\bigr)
			+\frac{\partial}{\partial Y}\bigl(X^aY^bY'\bigr)=
			bX^{a+1}Y^{b-1}
			+\frac{(b+1)(p-m)}{\sigma+2}X^{a+1}Y^b.
	\end{equation*}
	Since
	\begin{equation*}
		b<0,
		\qquad
		b+1=-\dfrac{2m}{mN-N+2}<0,
	\end{equation*}
	we readily obtain that $D(X,Y)<0$ throughout the first quadrant $\{X>0,\ Y>0\}$. Therefore, Dulac's criterion implies that system \eqref{systZ0} does not possess periodic orbits contained in this region. Moreover, since the axis $\{X=0\}$ is invariant and the flow across the axis $\{Y=0\}$ has a fixed direction, no periodic orbit can intersect the boundary of the first quadrant. Hence there are no closed orbits in the closure of $\{X>0,\ Y>0\}$.	
	
	\medskip
	
\noindent \textbf{Step 3. Trajectories from $P_0$ and $Q_5$.} We show next that the trajectory emerging from $P_0$ remains confined in a compact subset of the first quadrant. Arguing by contradiction, let us assume that this trajectory reaches infinity, that is, either $X(\eta)\to\infty$ or $Y(\eta)\to\infty$ (or both) as $\eta\to\infty$. Since the critical point $Q_2$, characterized as the limit of the trajectories for which
	$$
	\lim\limits_{\eta\to-\infty}Y(\eta)=\infty, \quad \lim\limits_{\eta\to-\infty}\frac{Y(\eta)}{X(\eta)}=\infty
	$$
	has opposite stability, we easily infer that $Y/X$ remains bounded along the trajectory. This fact implies that the $X(\eta)\to\infty$ as $\eta\to\infty$ and thus the part 
	$$
	\{(X(\eta),Y(\eta)): \eta>\eta_0>-\infty\}
	$$
	remains confined in a compact subset of the first quadrant of the system \eqref{systz0} after applying the change of variable \eqref{change2}. The non-existence of periodic orbits and the Poincaré--Bendixson's theorem yield that the trajectory $l_0$ must end in a finite critical point of the system \eqref{systz0} with $x=0$, that is, either $Q_1$ or $Q_5$. But the latter has opposite stability according to Lemma \ref{lem.Q1QgQ5} (its unique stable manifold is contained in the $y$-axis) and for the former, the trajectories reaching it arrive from the quadrant $\{X>0, Y<0\}$ according to the expansion \eqref{centerQ1}. We thus arrive to a contradiction, showing that the trajectory $l_0$ stays forever in a bounded set of the first quadrant. Another application of the Poincaré--Bendixson's theorem then readily gives that $l_0$ is a heteroclinic orbit connecting $P_0$ to the critical point $P_3$, as claimed.
	
	Actually, similar arguments show that every trajectory contained in the first quadrant $\{X>0,\ Y>0\}$ converges to $P_3$. Since the unstable manifold of $Q_5$ enters the interior of this region, it follows that the corresponding orbit is a heteroclinic connection from $Q_5$ to $P_3$. This proves the first two statements of the proposition.
	
	\medskip
	
	\noindent \textbf{Step 4. Trajectory entering $Q_1$.} It remains to determine the origin of the unique orbit entering $Q_1$ through its center manifold. By the expansion \eqref{centerQ1}, we have
	\begin{equation*}
		y=-\frac{\sigma+2}{p-m}x+o(x),
	\end{equation*}
	and therefore $y<0$ along every trajectory approaching $Q_1$ through its center manifold. Since $X>0$ and	$y=Y/X$, it follows that such trajectories satisfy $Y<0$ in a	neighborhood of $Q_1$ and, according to the previous argument, the trajectory reaching $Q_1$ and contained in the plane $\{Z=0\}$ will remain forever in the closure of the quadrant $\{X>0, Y<0\}$. We now follow this orbit backwards in time. Since
	\begin{equation*}
		\dot X=X(2-(m-1)Y)>0
	\end{equation*}
	in the region $\{X>0,\ Y<0\}$, the mapping $\eta\mapsto X(\eta)$ is strictly increasing along the trajectory. By the first part of the proof, the unique orbit emerging from $P_0$ is a heteroclinic connection from $P_0$ to $P_3$, whence the trajectory reaching $Q_1$ inside the invariant plane $\{Z=0\}$ cannot start from $P_0$. Moreover, trajectories starting from $Q_2$ are confined in the closure of the first quadrant $\{X>0, Y>0\}$ and cannot reach $Q_1$. Regarding the trajectory entering $Q_1$ in variables $(X,Y)$, we deduce from the monotonicity of $X(\eta)$, the opposite stability of $Q_3$, and similar arguments as in Step 3 that the critical point $P_1$ is the only possible origin of the orbit. Therefore, the unique orbit entering $Q_1$ through its center manifold is a heteroclinic connection from $P_1$ to $Q_1$, completing the proof.
\end{proof}

The following preparatory result concerns the dynamics on the invariant plane $\{x=0\}$ of the compactified system \eqref{PSinf1}, which corresponds to trajectories approaching infinity in the $X$-direction. In order to analyze the flow on this plane, we introduce the variable
\begin{equation*}
	w=xz\geq0,
\end{equation*}
which removes the singularity associated with the variable $Z$. Setting $x=0$ in the resulting system, we obtain the reduced system
\begin{equation}\label{PSsyst3.ext}
	\left\{
	\begin{array}{ll}
		\dot{y}=-y^2+\dfrac{p-m}{\sigma+2}y-w,\\[2mm]
		\dot{w}=(p+m-2)w,
	\end{array}
	\right.
\end{equation}
where derivatives are taken with respect to the variable $\eta_1$ introduced in \eqref{changexi}. The system \eqref{PSsyst3.ext} possesses the two critical points
\begin{equation*}
	Q_1'=(0,0),
	\qquad
	Q_5'=\left(\dfrac{p-m}{\sigma+2},0\right).
\end{equation*}
These critical points contain the traces on $\{x=0\}$ of the critical points $Q_1$ and $Q_5$ of the full system \eqref{PSinf1}, but, due to the change of variable $w=xz$, they may contain some more trajectories that, in variables $(y,z)$, might have $z\to\infty$. Their local analysis is completely identical to that performed in \cite[Lemma 4.3]{IS25b}, and therefore we only state the resulting properties.
\begin{proposition}\label{prop.Q4.bis}
	The critical point $Q_5'$ is a saddle point of the system \eqref{PSsyst3.ext}, whereas $Q_1'$ is a saddle-node. The unique orbit contained in the unstable manifold of $Q_5'$
	connects $Q_5'$ to the stable node $Q_3$, with the mapping $\eta_1\mapsto y(\eta_1)$ strictly decreasing. Every orbit emerging from $Q_1'$ converges to $Q_3$ as well and, along such trajectories, the function $\eta_1\mapsto y(\eta_1)$, admits one local maximum.
\end{proposition}

Our next preparatory result proves that, if a trajectory of the system \eqref{PSsyst} remains forever in the half-space $\{Y<-(\sigma+2)/(p-m)\}$, then it either enters the critical point $Q_1$ and corresponds to the profile of a peaking solution or reaches the attractor $Q_3$. 
\begin{proposition}\label{prop.left}
	Let $m$, $p$ and $\sigma$ as in \eqref{range.exp} and assume that $(X,Y,Z)(\eta)$ is a trajectory of the system \eqref{PSsyst} such that 
\begin{equation}\label{condpeak}
	Y(\eta)<-\frac{\sigma+2}{p-m}, \quad \eta\in(\eta_0,\infty).
\end{equation}
Then the trajectory either reaches the stable node $Q_3$ or coincides with one of the trajectories $\mathcal{C}_k$ for $k\in(0,Z(P_2))$ defined in Remark \ref{rem.Ck}.
\end{proposition}
\begin{proof}
We infer from the first and third equations of the system \eqref{PSsyst} that the mappings
$$
\eta\mapsto X(\eta), \quad \eta\mapsto Z(\eta), \quad \eta\in(\eta_0,\infty)
$$
are, respectively, increasing and decreasing. Hence the following limits exist: 
$$
X_{\infty}=\lim\limits_{\eta\to\infty}X(\eta), \quad Z_{\infty}=\lim\limits_{\eta\to\infty}Z(\eta)\in[0,Z(\eta_0)).
$$
Assume first that $X_{\infty}<\infty$. With respect to $Y(\eta)$, we have two possibilities:

$\bullet$ either $Y(\eta)$ is bounded for $\eta\in(\eta_0,\infty)$. If $Y(\eta)$ is convergent to some $Y_{\infty}\in(-\infty,-(\sigma+2)/(p-m)]$ we reach a contradiction, since $(X_{\infty},Y_{\infty},Z_{\infty})$ has to be a finite critical point of the system \eqref{PSsyst} with $X_{\infty}>0$ and $Y_{\infty}\leq-(\sigma+2)/(p-m)$ and there is no such point. Otherwise, there are $\{\eta_j^m\}_{j\geq1}$ and $\{\eta_j^M\}_{j\geq1}$ sequence of points of minimum, respectively maximum, of $Y(\eta)$, such that 
$$
\lim\limits_{j\to\infty}Y(\eta_j^m)=Y_{\infty}^m<Y_{\infty}^M=\lim\limits_{j\to\infty}Y(\eta_j^M).
$$
Thus, the $\omega$-limit set of the trajectory is the line segment 
$$
\{X=X_{\infty},  Y_{\infty}^m\leq Y\leq Y_{\infty}^M, Z=Z_{\infty}\}
$$
and, according to \cite[Theorem 2 and Corollary, Section 3.2]{Pe}, this line segment should be invariant for the flow of the system \eqref{PSsyst}. We reach a contradiction, as there is no such invariant horizontal segment in the finite part of the phase plane.

$\bullet$ or $Y(\eta)$ is unbounded for $\eta>\eta_0$, which implies that there is a sequence $\{\eta_j\}_{j\geq1}$ such that 
$$
\lim\limits_{j\to\infty}\eta_j=\infty, \quad \lim\limits_{j\to\infty}Y(\eta_j)=-\infty.
$$
The boundedness of $X_{\infty}$ and $Z_{\infty}$ implies that $(X,Y,Z)(\eta_j)$ converges to $Q_3$ as $j\to\infty$. Since $Q_3$ is a stable node, it follows that the trajectory reaches $Q_3$, which is one of the two claimed possibilities.

We are left now with the case $X_{\infty}=\infty$. Assume that the trajectory does not reach the stable node $Q_3$. This assumption entails that $(Y/X)(\eta)$ remains bounded for $\eta\in(\eta_0,\infty)$. We then pass to the $(x,y,z)$ variables of the system \eqref{PSinf1} and the trajectory under consideration is transformed into a trajectory on which $x=1/X\to0$, $z=Z/X\to0$ and $y=Y/X$ is bounded. If $y(\eta_1)$ converges as $\eta_1\to\infty$, then the trajectory ends up at a critical point of the form $(0,y_{\infty},0)$ with $-\infty<y_{\infty}\leq0$, and the only such point is $Q_1$. If $y(\eta_1)$ oscillates infinitely many times, we employ a similar argument as above with sequences of maxima and minima, and by evaluating the second equation of the system \eqref{PSinf1} at these maxima and minima and passing to the limit, we find that 
\begin{equation}\label{limmaxmin}
0=-(y_{\infty}^m)^2+\frac{p-m}{\sigma+2}y_{\infty}^m=-(y_{\infty}^M)^2+\frac{p-m}{\sigma+2}y_{\infty}^M,
\end{equation}
where
$$
y_{\infty}^m:=\liminf\limits_{\eta_1\to\infty}y(\eta_1), \quad y_{\infty}^M:=\limsup\limits_{\eta_1\to\infty}y(\eta_1).
$$
We deduce from \eqref{limmaxmin} that $y_{\infty}^m=y_{\infty}^M=0$ and thus the trajectory reaches the critical point $Q_1$. Finally, since the trajectory arrives at $Q_1$ from the half-space $\{Y<-(\sigma+2)/(p-m)\}$, we readily deduce from the expansion \eqref{centerQ1} and Remark \ref{rem.Ck} that it coincides with a trajectory $\mathcal{C}_k$ with $k\in(0,Z(P_2))$, completing the proof.
\end{proof}
We close this section with an analogous result analyzing trajectories remaining in the region limited by the planes $\{Y=2/(m-1)\}$ and $\{Y=-(\sigma+2)/(p-m)\}$. 
\begin{proposition}\label{prop.ret}
	Let $(X,Y,Z)(\eta)$ be a trajectory of the system \eqref{PSsyst} such that there exists $\eta_0\in\real$ with $X(\eta_0)>0$, $Z(\eta_0)>0$ and
	\begin{equation}\label{interm22}
		-\frac{\sigma+2}{p-m}<Y(\eta)<\frac{2}{m-1}, \qquad {\rm for \ any} \ \eta\in(\eta_0,\infty).
	\end{equation}
	Then, this trajectory enters either the critical point $Q_1$ or $Q_{\gamma_0}$ with $\gamma_0$ defined in Lemma \ref{lem.Q1QgQ5}.
\end{proposition}
The proof follows very similar arguments as the one of Proposition \ref{prop.left}. However, we omit the details and refer the reader to \cite[Lemma 4.5]{IS25b} where a complete and detailed proof is given. 

\section{Proof of Theorem \ref{th.pSpJL}}\label{sec.pSpJL}

We assume throughout this section that $p_S(\sigma)<p<p_{JL}(\sigma)$. Let us recall here that, in this interval, the restriction of the critical point $P_2$ to the invariant plane $\{X=0\}$ is a stable focus. The proof of Theorem \ref{th.pSpJL} is then based on a strategy of shooting on the unstable manifold of the critical point $P_0$. Indeed, we infer from Proposition \ref{prop.X0} that the unique trajectory contained in the unstable manifold of $P_0$ and in the plane $\{X=0\}$ (denoted by $l_{\infty}$ in the notation of Lemma \ref{lem.P0P1P3}) is a heteroclinic trajectory between $P_0$ and $P_2$, while Proposition \ref{prop.Z0} entails that the unique trajectory contained in the unstable manifold of $P_0$ and in the plane $\{Z=0\}$ (denoted by $l_0$ in the notation of Lemma \ref{lem.P0P1P3}) is a heteroclinic trajectory between $P_0$ and $P_3$. We divide the proof into several steps.

\medskip 

\noindent \textbf{Step 1. Definition of the three sets.} Assume first that the unique trajectory contained in the unstable manifold of $P_3$ (according to Lemma \ref{lem.P0P1P3}) is a heteroclinic one connecting to the stable node $Q_3$. The second equation of the system \eqref{PSsyst} readily gives (by directly replacing $Y=-(\sigma+2)/(p-m)$ in it) that the direction of the flow of the system \eqref{PSsyst} across the plane $\{Y=-(\sigma+2)/(p-m)\}$ points in the negative direction of $Y$ at points with $Z>Z(P_2)$ and in the positive direction of $Y$ at points with $Z<Z(P_2)$. Thus, a trajectory connecting $P_3$ to $Q_3$ must have an odd number of intersections with the plane $\{Y=-(\sigma+2)/(p-m)\}$, denoted by $2j+1$, where $j=0,1,2,...$ is a non-negative integer. Moreover, no trajectory can be tangent to the plane $\{Y=-(\sigma+2)/(p-m)\}$: the only possible tangency points, satisfying $\dot{Y}=0$, have $Z=Z(P_2)$ and thus belong to the trajectory \eqref{estacionariaLine}. By the uniqueness of a trajectory through a non-critical point of the system \eqref{PSsyst}, we infer that such a tangency is not possible.

We next split the trajectories $\{l_k\}_{k\in(0,\infty)}$ contained in the unstable manifold of $P_0$ according to the notation in Lemma \ref{lem.P0P1P3} into the following three sets: 
\begin{equation}\label{sets0}
\begin{split}
	&\mathcal{E}_1=\left\{k\in(0,\infty): l_k \ {\rm crosses} \ Y=-\frac{\sigma+2}{p-m} \ {\rm at \ least} \ 2j+2 \ {\rm times}\right\},\\
	&\mathcal{G}_1=\left\{k\in(0,\infty): l_k \ {\rm crosses} \ Y=-\frac{\sigma+2}{p-m} \ {\rm exactly} \ 2j+1 \ {\rm times \ and \ enters} \ Q_3\right\},\\[2mm]
	&\mathcal{F}_1=(0,\infty)\setminus(\mathcal{E}_1\cup\mathcal{G}_1).
\end{split}
\end{equation}
The definitions of $\mathcal{E}_1$ and $\mathcal{G}_1$, together with the direction of the flow of the system \eqref{PSsyst} across the plane $\{Y=-(\sigma+2)/(p-m)\}$ described in the previous paragraph and the stability of $Q_3$ entail that $\mathcal{E}_1$ and $\mathcal{G}_1$ are open sets. 

\medskip 

\noindent \textbf{Step 2. Existence of a peaking solution.} Since $p\in(p_{S}(\sigma),p_{JL}(\sigma))$, Lemma \ref{lem.P2} ensures that $P_2$ is a stable focus for the reduced system \eqref{systX0}, thus the trajectory connecting $P_0$ to $P_2$ according to Proposition \ref{prop.X0} has infinitely many intersections with the plane $\{Y=-(\sigma+2)/(p-m)\}$. An obvious continuity argument implies that there is $\vartheta>0$ such that the trajectory $l_{k}$ with $k\in(\vartheta,\infty)$ has at least $2j+2$ intersection points with the same plane; that is, $(\vartheta,\infty)\subseteq\mathcal{E}_1$. On the contrary, since the unique trajectory contained in the unstable manifold of $P_3$ has exactly $2j+1$ intersection points with the plane $\{Y=-(\sigma+2)/(p-m)\}$ and then enters $Q_3$, we infer from Proposition \ref{prop.Z0}, the stability of $Q_3$ and the local behavior near a saddle point \cite[Theorem 2.9]{Shilnikov} that there is $\varepsilon>0$ such that the trajectories $l_{k}$ with $k\in(0,\varepsilon)$ remain in a tubular neighborhood of the "composition" of the trajectory $l_0$ followed by the unique trajectory stemming from $P_3$ and thus enter $Q_3$ after intersecting $2j+1$ times the plane $\{Y=-(\sigma+2)/(p-m)\}$; that is, $(0,\varepsilon)\subseteq\mathcal{G}_1$. 

A standard topological argument then gives that $\mathcal{F}_1$ is nonempty. In particular, letting $k_1=\inf\,\mathcal{E}_1$, the openness of both $\mathcal{E}_1$ and $\mathcal{G}_1$ implies that $k_1\in\mathcal{F}_1$. Thus, the trajectory $l_{k_1}$ has at most $2j+1$ points of intersection with the plane $\{Y=-(\sigma+2)/(p-m)\}$. We deduce from Propositions \ref{prop.left} and \ref{prop.ret} that the trajectory $l_{k_1}$ ends at one of the critical points $Q_1$, $Q_3$ or $Q_{\gamma_0}$. In all these cases, by eventually applying \cite[Theorem 5.11]{Shilnikov} (dealing with the local behavior near a non-hyperbolic point with a saddle-type behavior) and Proposition \ref{prop.Q4.bis} when $l_{k_1}$ enters $Q_1$ or $Q_{\gamma_0}$, we deduce that there is $\delta>0$ such that, for any $k\in(k_1,k_1+\delta)$, the trajectory $l_k$ has
	
$\bullet$ either exactly the same number of oscillations with respect to the plane $Y=\{-(\sigma+2)/(p-m)\}$, if $l_{k_1}$ enters $Q_3$,

$\bullet$ or one more oscillation than $l_{k_1}$ if $l_{k_1}$ reaches $Q_{\gamma_0}$ or $Q_1$ from the region $\{Y>-(\sigma+2)/(p-m)\}$,

$\bullet$ or two more oscillations than $l_{k_1}$ if $l_{k_1}$ reaches $Q_1$ from the region  $\{Y<-(\sigma+2)/(p-m)\}$, in which case $l_{k_1}$ is already a trajectory corresponding to a peaking solution.

By the definition of infimum, $(k_1,k_1+\delta)\cap\mathcal{E}_1\neq\emptyset$, hence we readily deduce from the previous argument and the definition of $\mathcal{E}_1$ that $l_{k_1}$ must have exactly $2j+1$ points of intersection with the plane $\{Y=-(\sigma+2)/(p-m)\}$. Consequently, $l_{k_1}$ does not reach the critical point $Q_3$, as $k_1\not\in\mathcal{G}_1$. 

Let $\eta_0\in\real$ be the ``time" of the last intersection point with the plane $\{Y=-(\sigma+2)/(p-m)\}$, that is, $Y(\eta_0)=-(\sigma+2)/(p-m)$ and $Y(\eta)\neq-(\sigma+2)/(p-m)$ for any $\eta>\eta_0$. The odd number of intersection points entails that $Y(\eta)<-(\sigma+2)/(p-m)$ for any $\eta\in(\eta_0,\infty)$. An application of Proposition \ref{prop.left} guarantees that the trajectory $l_{k_1}$ coincides with a trajectory $\mathcal{C}_s$ for some $s\in(0,Z(P_2))$. The existence of a peaking solution is then ensured by Theorem \ref{th.peaking}.

\medskip

\noindent \textbf{Step 3. Towards a second solution.} The multiplicity of solutions follows by an induction similar to the one employed in the proof of \cite[Theorem 1.1]{IS25b}. We just give here the step of deriving the second solution, as the continuation of the inductive argument is then straightforward and left to the reader. We show first that there is $\delta>0$ and $k_2\in(k_1,k_1+\delta)$ such that $k_2\in\mathcal{E}_1$ and the trajectory $l_{k_2}$ reaches the stable node $Q_3$. According to the definition of infimum, the openness of $\mathcal{E}_1$ and the generalized boundary value theorem \cite[Theorem 5.11]{Shilnikov} imply that there exists $\delta>0$ such that $(k_1,k_1+\delta)\subseteq\mathcal{E}_1$ and the trajectories $l_k$ with $k\in(k_1,k_1+\delta)$ reach the stable node $Q_3$. Indeed, we infer from \cite[Theorem 5.11]{Shilnikov} that the trajectories $l_k$ with $k$ very close to $k_1$ remain in a tubular neighborhood of the trajectories lying on the unstable manifold of $Q_1$ contained in the plane $\{x=0\}$. According to Proposition \ref{prop.Q4.bis}, these trajectories add up (with respect to the trajectory $l_{k_1}$) exactly two more oscillations with respect to the plane $\{Y=-(\sigma+2)/(p-m)\}$ (that is, a total of $2j+3$ intersection points) before reaching the attractor $Q_3$. 

\medskip

\noindent \textbf{Step 4. The inductive step and existence of a second peaking solution.} Thus, for a second trajectory corresponding to a peaking solution, we split the open set $\mathcal{E}_1$ into three sets as indicated below:
\begin{equation}\label{sets1}
	\begin{split}
		&\mathcal{E}_2=\left\{k\in\mathcal{E}_1: l_k \ {\rm crosses} \ Y=-\frac{\sigma+2}{p-m} \ {\rm at \ least} \ 2j+4 \ {\rm times}\right\},\\
		&\mathcal{G}_2=\left\{k\in\mathcal{E}_1: l_k \ {\rm crosses} \ Y=-\frac{\sigma+2}{p-m} \ {\rm exactly} \ 2j+3 \ {\rm times \ and \ enters} \ Q_3\right\},\\[2mm]
		&\mathcal{F}_2=\mathcal{E}_1\setminus(\mathcal{E}_2\cup\mathcal{G}_2).
	\end{split}
\end{equation}
It is obvious that $\mathcal{E}_2$ and $\mathcal{G}_2$ are open sets and the arguments in the previous step ensure that $\mathcal{G}_2$ contains an interval $(k_1,k_1+\delta)$ such that the trajectories $l_k$ with $k\in(k_1,k_1+\delta)$ have exactly $2j+3$ oscillations with respect to the plane $\{Y=-(\sigma+2)/(p-m)\}$. Moreover, since $p\in(p_{S}(\sigma),p_{JL}(\sigma))$, we remind the reader that the trace of the critical point $P_2$ in the invariant plane $\{X=0\}$ is a stable focus by Lemma \ref{lem.P2} and thus the trajectory $l_0$ has infinitely many intersection points with the plane $\{Y=-(\sigma+2)/(p-m)\}$. An argument of continuity implies then that $\mathcal{E}_2$ contains an interval of the form $(\varepsilon,\infty)$. In particular, both $\mathcal{E}_2$ and $\mathcal{G}_2$ are non-empty and open sets and, since $\mathcal{E}_1$ is also open, it follows that there exists $k_2\in\mathcal{F}_2$. Similar arguments as at the end of Step 2 then guarantee that $l_{k_2}$ is a trajectory corresponding to a profile (after undoing the transformation \eqref{PSchange}) of a peaking solution. This inductive construction can be then extended infinitely many times in an obvious way (by splitting next the open set $\mathcal{E}_2$ as above), concluding the proof under the assumption that the unique trajectory in the unstable manifold of $P_3$ reaches $Q_3$. 

\medskip 

\noindent \textbf{Step 5. Removing the initial assumption.} We are left with the case when the trajectory in the unstable manifold of $P_3$ does not reach the critical point $Q_3$. If the trajectory has only a finite number of oscillations with respect to the plane $\{Y=-(\sigma+2)/(p-m)\}$, one can readily conclude that there is $\eta_0\in\real$ such that, for $\eta>\eta_0$, either \eqref{interm22} or \eqref{condpeak} holds true and thus the trajectory reaches either one of the critical points $Q_1$ or $Q_{\gamma_0}$ according to Propositions \ref{prop.left} and \ref{prop.ret}. In both cases, it follows exactly as in Step 3 from the generalized boundary value problem \cite[Theorem 5.11]{Shilnikov} and Proposition \ref{prop.Q4.bis} that there is $\varepsilon_0>0$ such that, for any $k\in(0,\varepsilon_0)$, the trajectory $l_k$ reaches the critical point $Q_3$. One can then apply the Steps 1-4 by splitting into three sets the interval $(\varepsilon_0/2,\infty)$ instead of $(0,\infty)$ to complete the proof.

\medskip 

\noindent \textbf{Step 6. Infinitely many oscillations.} On the contrary, if the trajectory oscillates infinitely many times with respect to the plane $\{Y=-(\sigma+2)/(p-m)\}$, then there are sequences $\{\eta_j^m\}_{j\geq1}$ and $\{\eta_j^M\}_{j\geq1}$ of maxima and minima of the mapping $\eta\mapsto Y(\eta)$ along the trajectory such that
$$
\lim\limits_{j\to\infty}\eta_j^m=\lim\limits_{j\to\infty}\eta_j^M=\infty, \quad Y(\eta_j^m)<-\frac{\sigma+2}{p-m}<Y(\eta_j^M), \quad j\geq1.
$$
Assume first that there is $j_0\geq1$ such that $Y(\eta_j^M)<2/(m-1)$ for any $j\geq j_0$. We infer that $\eta\mapsto X(\eta)$ is an increasing mapping along the trajectory for $\eta\in(\eta_j^M,\infty)$ and thus there exists $X_{\infty}=\lim\limits_{\eta\to\infty}X(\eta)$. If $X_{\infty}<\infty$, it follows that the $\omega$-limit set is either a critical point or a trajectory of the system \eqref{PSsyst} contained in the plane $\{X=X_{\infty}\}$ according to \cite[Theorem 2 and Corollary, Section 3.2]{Pe}. But the latter is impossible, as follows from the analysis performed in Section \ref{sec.local}. We thus deduce that $X_{\infty}=\infty$ and, by passing to the system \eqref{PSinf1} in variables $(x,y,z)$, we can easily deduce that the trajectory reaches a critical point contained in the invariant plane $\{x=0\}$ of the system \eqref{PSinf1}, as that plane does not admit periodic orbits according to Proposition \ref{prop.Q4.bis}. 

We are thus left with the case when (after eventually extracting subsequences, not relabeled) 
\begin{equation}\label{osc}
Y(\eta_j^m)<-\frac{\sigma+2}{p-m}<\frac{2}{m-1}<Y(\eta_j^M), \quad j\geq1.
\end{equation}
We show next that this situation is not possible. Fixing $\eta_0\in\real$, we infer by integration from the first and third equations of the system \eqref{PSsyst} that 
\begin{equation}\label{limits}
\begin{split}
&\frac{\ln X(\eta)-\ln X(\eta_0)}{\eta-\eta_0}=2-\frac{m-1}{\eta-\eta_0}\int_{\eta_0}^{\eta}Y(s)\,ds, \\ &\frac{\ln Z(\eta)-\ln Z(\eta_0)}{\eta-\eta_0}=\sigma+2+\frac{p-m}{\eta-\eta_0}\int_{\eta_0}^{\eta}Y(s)\,ds.
\end{split}
\end{equation}
Assume for contradiction that $X(\eta)$ is bounded along the trajectory. We then pass to the limit as $\eta\to\infty$ to obtain
$$
\lim\limits_{\eta\to\infty}\frac{1}{\eta-\eta_0}\int_{\eta_0}^{\eta}Y(s)\,ds=\frac{2}{m-1},
$$
which also gives, after letting $\eta\to\infty$ in the second equation in \eqref{limits}, that 
$$
\lim\limits_{\eta\to\infty}\frac{\ln Z(\eta)-\ln Z(\eta_0)}{\eta-\eta_0}=\sigma+2+\frac{2(p-m)}{m-1}=\frac{\sigma(m-1)+2(p-1)}{m-1}>0.
$$
The latter implies that $Z(\eta)\to\infty$ as $\eta\to\infty$, which is a contradiction with the infinite oscillation \eqref{osc} around the plane $\{Y=-(\sigma+2)/(p-m)\}$ (which implies that there are subsequences on which $\eta\mapsto Z(\eta)$ decreases). We have thus proved that $X(\eta)$ is unbounded. Passing to the variables $(x,y,z)$ of the system \eqref{PSinf1}, we deduce that the $\omega$-limit set of the trajectory contains points in the invariant plane $\{x=0\}$ and, by the invariance of this plane and the fact that the $\omega$-limit set is itself a critical point or a trajectory of the system (see \cite[Section 3.2]{Pe}), we conclude that the $\omega$-limit set of the trajectory is contained in the invariant plane $\{x=0\}$. The stability of the critical points $Q_3$ and $Q_2$ gives that $|y|=|Y|/X$ remains bounded along the trajectory. But then the condition \eqref{osc} and the invariance of the $z$-axis in the plane $\{x=0\}$ ensures that the $\omega$-limit set must be contained in the $z$-axis. We then deduce from Lemma \ref{lem.Q1QgQ5} and an argument of oscillations with respect to $z$ given in full detail in \cite[Lemma 4.5]{IS25b} that the limit is necessarily a critical point (either $Q_1$ or $Q_{\gamma_0}$) and we conclude as in Step 5.

A visual presentation of this proof throughout numerical experiments is given in Figure \ref{fig1} in Section \ref{sec.num} below.

\section{The main estimates}\label{sec.calc}

This section contains the main analytical ingredient in the proof of Theorems \ref{th.sigmasmall} and \ref{th.sigmalarge}. The strategy is to identify sharp parameter ranges for which the solutions of a suitable linearized equation have at least three zeros. As it will become clear from the subsequent phase-space analysis, these zeros correspond to crossings of the plane $\{Z=Z(P_2)\}$ by trajectories approaching the critical point $Q_1$. This information will ultimately allow us to construct backward self-similar profiles satisfying \eqref{beh.inf}.

The calculations below are inspired by the approach introduced by Lepin \cite{L89,L90} and later adapted to the quasilinear setting by Galaktionov and V\'azquez \cite[Section 12]{GV97} in the derivation of the critical exponents \eqref{pL}. Let us mention here that some of the subsequent calculations have been performed with the assistance of the symbolic computing Maple software.

Since our aim is to look for profiles exhibiting the asymptotic behavior \eqref{beh.inf} as $\xi\to\infty$, we perform suitable transformations with the goal of transforming the tail behavior \eqref{beh.inf} to a constant (but positive) limit in the new variables and, at the same time, to map the stationary solution \eqref{estacionaria} into a constant solution to the transformed equation. We start from \eqref{backward.SS} and introduce the following changes of variable: in a first step set
\begin{equation*}
	h(\xi)=f^m(\xi),
\end{equation*}
which transforms \eqref{ODE.backward} into
\begin{equation}\label{ODE.h}
	h''(\xi)+\dfrac{N-1}{\xi}h'(\xi)-\alpha h^{1/m}(\xi)-\beta\xi\left(h^{1/m}\right)'(\xi)+\xi^\sigma h^{p/m}(\xi)=0.
\end{equation}
Next, we perform the transformation
\begin{equation*}
	h(\xi)=\xi^{-m(\sigma+2)/(p-m)}g(\xi),\qquad \xi=e^s, \qquad s\in\mathbb{R}.
\end{equation*}
A straightforward computation yields
\begin{equation*}
	\begin{split}
		h'(\xi)=&\ \xi^{-m(\sigma+2)/(p-m)}g'(\xi)-\dfrac{m(\sigma+2)}{p-m}\xi^{-m(\sigma+2)/(p-m)-1}g(\xi),\\[2mm]
		h''(\xi)=&\ \xi^{-m(\sigma+2)/(p-m)}g''(\xi)-\dfrac{2m(\sigma+2)}{p-m}\xi^{-m(\sigma+2)/(p-m)-1}g'(\xi)\\
		&+\dfrac{m(\sigma+2)}{p-m}\left(\dfrac{m(\sigma+2)}{p-m}+1\right)\xi^{-m(\sigma+2)/(p-m)-2}g(\xi).
	\end{split}
\end{equation*}
Substituting these expressions into \eqref{ODE.h}, we obtain
\begin{equation*}
	\begin{split}
		\xi^2g''(\xi)&+\left(N-1-\dfrac{2m(\sigma+2)}{p-m}\right)\xi g'(\xi)-\dfrac{m(\sigma+2)}{p-m}\left(N-2-\dfrac{m(\sigma+2)}{p-m}\right)g(\xi)\\[2mm]
		&-\dfrac{\beta}{m}\xi^{1/\beta+1}g^{1/m-1}(\xi)g'(\xi)+g^{p/m}(\xi)=0.
	\end{split}
\end{equation*}
In terms of the variable $s=\ln\xi$, this equation becomes
\begin{equation}\label{ODE.g}
	g''(s)+Ag'(s)-\dfrac{m}{p-m}Bg(s)+g^{p/m}(s)-\dfrac{\beta}{m}e^{s/\beta}g^{1/m-1}(s)g'(s)=0,
\end{equation}
where
\begin{equation}\label{AB}
	A=N-2-\dfrac{2m(\sigma+2)}{p-m},\qquad B=(\sigma+2)\left(N-2-\dfrac{m(\sigma+2)}{p-m}\right).
\end{equation}
Furthermore, the stationary profile \eqref{estacionaria} corresponds to the constant solution $g\equiv C(\sigma)$. We therefore linearize \eqref{ODE.g} around this equilibrium by writing
\begin{equation*}
	g(s)=C(\sigma)+y(s).
\end{equation*}
The resulting linearized equation is
\begin{equation}\label{ylinear}
	y''(s)+Ay'(s)+By(s)-\dfrac{\beta}{m}C^{1/m-1}(\sigma)e^{s/\beta}y'(s)=0.
\end{equation}
Equation \eqref{ylinear} falls within the framework studied by Lepin \cite{L89}. In particular, if
\begin{equation*}
	\dfrac{2(j+1)}{\beta}<A-\sqrt{A^2-4B}\leq\dfrac{2(j+2)}{\beta}, \quad j=0,1,2,...
\end{equation*}
then the solution satisfying $y(s)\to1$ as $s\to\infty$ has exactly
$j+2$ zeros. By linearity, the same conclusion holds for any solution
satisfying $y(s)\to K$ with $K>0$.

Since our subsequent phase-space argument requires (as we shall see in the next section) at least three crossings of the plane $\{Z=Z(P_2)\}$, we are naturally led to seek solutions of \eqref{ylinear} having at least three zeros. This corresponds to the case $j=1$, which gives
\begin{equation}\label{condAB}
	A-\sqrt{A^2-4B}-\dfrac{4}{\beta}>0.
\end{equation}
The remainder of this subsection is devoted to translating condition \eqref{condAB} into explicit restrictions on the parameters of Eq. \eqref{eq1} and is divided into several steps and cases.

\medskip

\noindent \textbf{Step 1. We need $N>10$.} Observe first that the condition $p\geq p_{JL}(\sigma)$ is equivalent to $A^2-4B\geq0$, with equality attained exactly when $p=p_{JL}(\sigma)$. Assume from now on that $p\geq p_{JL}(\sigma)$ and $\sigma<\sigma_1(N)$, since the opposite case has been already solved in Section \ref{sec.pSpJL} and, by \eqref{pJL}, if $\sigma\geq\sigma_1(N)$ then $p_{JL}(\sigma)=\infty$. Observing that 
\begin{equation}\label{funcp}
A-\frac{4}{\beta}=-\frac{f(p)}{p-m}, \quad f(p):=-(N-10)p+(N+6\sigma+2)m-4\sigma-8,
\end{equation}
we notice that $f(p)$ is a linear function of $p$ which is increasing if $N<10$ and decreasing if $N>10$. Since 
$$
f(p_S(\sigma))=\frac{4(\sigma+2)[N(m-1)+2(m+1)]}{N-2}>0,
$$
we infer that $f(p)>0$ for any $p>p_S(\sigma)$ if $N\leq10$, that is, $A-4/\beta<0$ if $p>p_S(\sigma)$ and $N\leq10$. Thus, if furthermore $p\geq p_{JL}(\sigma)$, the condition \eqref{condAB} is never fulfilled. We conclude that, in order to have at least three zeros, we need also to assume that $N>10$. From now on, we introduce the following critical exponent 
\begin{equation}\label{eq.p1}
	p_1:=\frac{m(N+6\sigma+2)-4(\sigma+2)}{N-10},
\end{equation}
which is the unique zero of the function $f$ defined in \eqref{funcp}. We observe that 
$$
p_1-p_S(\sigma)=\frac{4(\sigma+2)[N(m-1)+2(m+1)]}{(N-2)(N-10)}>0,
$$
whence $p_1>p_S(\sigma)$. Moreover, since $N>10$, the function $f$ is decreasing with respect to $p$, hence $f(p)>0$ if $p<p_1$ and $f(p)<0$ if $p>p_1$.

\medskip 

\noindent \textbf{Step 2. Comparison between $p_{JL}(\sigma)$ and $p_1$.} We deduce from the previous step that we next have to compare next $p_1$ with $p_{JL}(\sigma)$. To this end, we introduce the function
\begin{equation*}
	\begin{split}
		f_{JL}(p)&:=(A^2-4B)(p-m)^2=(N-2)(N-10-4\sigma)p^2\\
		&-2(N^2-2(\sigma+4)N-2\sigma^2-4\sigma+4)mp+m^2(N-2)^2,
	\end{split}
\end{equation*}
which obviously satisfies $f_{JL}(p_{JL}(\sigma))=0$. Since $\sigma<\sigma_1(N)$ according to Remark \eqref{rem.estacionaria}, the graph of $f_{JL}$ is a parabola with positive dominating coefficient. Since
$$
f(p_S(\sigma))=-\frac{8m^2(\sigma+2)^2}{N-2}<0
$$ 
and $p_S(\sigma)<p_{JL}(\sigma)$, we deduce that $p_{JL}(\sigma)$ is the largest zero of $f_{JL}$. Thus, in order to compare $p_1$ with $p_{JL}(\sigma)$, we have to study the sign of $f_{JL}(p_1)$, recalling that $p_1$ is defined in \eqref{eq.p1}. A direct calculation shows that 
$$
f_{JL}(p_1)=\frac{8(\sigma+2)^2}{(N-10)^2}Q(m),
$$
where
\begin{equation}\label{polQ}
Q(m):=(2N^2-15N\sigma-22N+6\sigma+20)(m-1)^2-8\sigma(N+2)(m-1)-N\sigma-2N-6\sigma+20.
\end{equation}
Let us recall the exponents $\sigma^*(N)$ and $\sigma_*(N)$ defined in \eqref{sigmastar}. We now observe the role played by them in the analysis: $\sigma^*(N)$ is the zero of the dominating coefficient of the polynomial $Q(m)$ introduced in \eqref{polQ}, while $\sigma_*(N)$ is the root (with respect to $\sigma$) of the zero order term in $Q(m)$ (which is equal to $Q(1)$). Denoting by 
$$
D(\sigma):=2N^2-15N\sigma-22N+6\sigma+20
$$
the dominating coefficient of $Q(m)$, we observe that $$
D(0)=2(N-10)(N-1)>0, \quad D(\sigma^*(N))=0, \quad D(\sigma_1(N))=-\frac{(N-10)(7N+2)}{4}<0,
$$
whence $0<\sigma^*(N)<\sigma_1(N)$, since $D$ is a linearly decreasing function of $\sigma$. It is obvious that $-2<\sigma_*(N)<0$. We split then the analysis of the sign of $Q(m)$ into several cases:

\medskip 

$\bullet$ if $\sigma\in(\sigma^*(N),\sigma_1(N))$, then $D(\sigma)<0$, hence all the coefficients of $Q(m)$ are negative (since $\sigma>0$ in this case) and thus $Q(m)<0$ for any $m>1$. Consequently, $f_{JL}(p_1)<0$ and, by basic properties of parabolas, we conclude that $p_1<p_{JL}(\sigma)$ in this case.

$\bullet$ if $-2<\sigma<\sigma_*(N)$, then $D(\sigma)>0$ and $Q(1)>0$. In this case, a direct calculation shows that the polynomial $Q(m)$ has two different real roots given exactly by the expressions of $m_1>1$ and $m_2>1$ defined in \eqref{m12}. It then follows that $Q(m)<0$ for $m_1<m<m_2$ and $Q(m)>0$ for $m\in(1,m_1)\cup(m_2,\infty)$.

$\bullet$ if $\sigma_*(N)\leq\sigma\leq\sigma^*(N)$ then $Q(1)<0$ and thus $m_1<1<m_2$. In this case, $Q(m)<0$ for $m\in(1,m_2)$ and $Q(m)>0$ for $m>m_2$.

\medskip 

Putting together the previous cases and recalling the expression of $f_{JL}$ and its properties, we conclude this step with the following alternative:
\begin{equation}\label{compp1pJL}
	\max\{p_1,p_{JL}(\sigma)\}=\left\{\begin{array}{ll}p_1, & {\rm if} \ -2<\sigma<\sigma_*(N) \ {\rm and} \ m\in(1,m_1)\cup(m_2,\infty),\\
	p_{JL}(\sigma), & {\rm if} \ -2<\sigma<\sigma_*(N) \ {\rm and} \ m\in[m_1,m_2],\\
	p_1, & {\rm if} \ \sigma\in[\sigma_*(N),\sigma^*(N)] \ {\rm and} \ m\in(m_2,\infty),\\
	p_{JL}(\sigma), & {\rm if} \ \sigma\in[\sigma_*(N),\sigma^*(N)] \ {\rm and} \ m\in(1,m_2],\\
	p_{JL}(\sigma), & {\rm if} \ \sigma\in(\sigma^*(N),\sigma_1(N)) \ {\rm and} \ m\in(1,\infty). \end{array}\right.
\end{equation}
Let us observe that, whenever $p\in[p_{JL}(\sigma),p_1)$ (an interval well defined only in a few cases according to \eqref{compp1pJL}), the condition \eqref{condAB} is not fulfilled and thus one can only guarantee the existence of at most two zeros of the solutions $y(s)$ to \eqref{ylinear}.

\medskip 

\noindent \textbf{Step 3.} Let us assume from now on that $p\geq\max\{p_1,p_{JL}(\sigma)\}$. In this case, we have $A-4/\beta\geq0$ and we can check the condition \eqref{condAB} by taking squares in order to remove the square root. We then get by direct calculation that 
\begin{equation}\label{comp0}
A^2-4B-\left(A-\frac{4}{\beta}\right)^2=\frac{4}{(p-m)^2}R(p),
\end{equation}
where 
\begin{equation*}
	\begin{split}
R(p)&:=[2(N-10)-\sigma(N-2)](p-m)^2+(\sigma+2)[m(2N+\sigma-26)-2(N-10)](p-m)\\
&-4(2m-1)(m-1)(\sigma+2)^2.
	\end{split}
\end{equation*}
It is obvious that $R(p_{JL}(\sigma))<0$ by definition. Setting $P:=(p-m)/(\sigma+2)$, we easily observe that 
\begin{equation}\label{polR}
\begin{split}
	\frac{R(p)}{(\sigma+2)^2}&=\overline{R}(P):=[2(N-10)-\sigma(N-2)]P^2\\&+[m(2N+\sigma-26)-2(N-10)]P
	-4(2m-1)(m-1).
\end{split}
\end{equation}
The remaining part of the analysis is focused on the sign of $R(p)$ and its simplified version $\overline{R}(p)$ for critical values of $p$, allowing to deduce the sign of the expression in \eqref{comp0} and thus the intervals where \eqref{condAB} is fulfilled.

\medskip 

\noindent \textbf{Step 4. The case $\sigma<\sigma_2(N)$.} In this case, the dominating coefficient of the polynomial $R$ is positive and, since $R(p_{JL}(\sigma))<0$, the polynomial $R$ has two real roots (and $p_{JL}(\sigma)$ lies between them). A simple calculation ensures that the largest root is given exactly by the exponent $p_L(\sigma)=p_L^{-}(\sigma)$ defined in \eqref{pL}. Then, $R(p)<0$ for $p\in(p_{JL}(\sigma),p_L(\sigma))$ and the condition \eqref{condAB} is satisfied whenever $p>p_1$ and $p\in(p_{JL}(\sigma),p_{L}(\sigma))$. Moreover, 
\begin{equation}\label{Rp}
R(p_1)=\frac{(p-m)^2}{4}(A^2-4B)>0.
\end{equation}
Since $p_L(\sigma)$ is the largest root of $R$ and $R(p_{JL}(\sigma))<0$, we deduce that $p_1>p_L(\sigma)$ whenever $p_1>p_{JL}(\sigma)$. Hence, whenever $p_1>p_{JL}(\sigma)$, the condition \eqref{condAB} is not fulfilled for any $p>p_{JL}(\sigma)$ and thus the solution $y$ to \eqref{ylinear} such that $y(s)\to K$ as $s\to\infty$ has only two zeros. 

Before completing the analysis of this case, we still have to compare $\sigma^*(N)$ with $\sigma_2(N)$. We thus have 
$$
\sigma_2(N)-\sigma^*(N)=-\frac{2(N-10)(N^2-18N+8)}{3(N-2)(5N-2)},
$$
which is non-negative for $10<N\leq9+\sqrt{73}$ and is negative for $N>9+\sqrt{73}$. In terms of integer dimensions, we find 
\begin{equation}\label{compsigma}
\min\{\sigma^*(N),\sigma_2(N)\}=\left\{\begin{array}{ll}\sigma^*(N), & {\rm if} \ 10<N<18,\\
\sigma_2(N), & {\rm if} \ N\geq18.\end{array}\right.
\end{equation}
We conclude from the analysis performed in Steps 2, 3 and 4 that the condition \eqref{condAB} is fulfilled if one of the following conditions is satisfied

$\bullet$ $\sigma\in(-2,\sigma_*(N))$ and $m\in[m_1,m_2]$,

$\bullet$ $11\leq N\leq 17$, $\sigma\in[\sigma_*(N),\sigma^*(N)]$ and $m\in(1,m_2]$, 

$\bullet$ $11\leq N\leq 17$, $\sigma\in(\sigma^*(N),\sigma_2(N))$ and $m>1$.

$\bullet$ $N\geq18$, $\sigma\in(\sigma_2(N),\sigma^*(N))$ and $m\in(1,m_2]$,

\noindent and in all these cases for $p\in[p_{JL}(\sigma),p_L(\sigma))$. We have thus derived the intervals given in Theorem \ref{th.sigmasmall}.

\medskip 

\noindent \textbf{Step 5. The range $\sigma_2(N)\leq\sigma<\sigma_1(N)$.} This case is qualitatively simpler, although technically more involved. We observe that the dominating coefficient of the polynomial $R$ is now negative, thus the graph of $R(p)$ is a parabola with negative quadratic coefficient. Noticing that, on the one hand $R(p_1)>0$ and $R(p_{JL}(\sigma))<0$ (exactly as in the previous case), and, on the other hand, the condition \eqref{condAB} is fulfilled whenever $R(p)<0$ and $p\geq p_{JL}(\sigma)$, we deduce by repeating the arguments of the previous step that the condition \eqref{condAB} is satisfied for 
\begin{equation}\label{lowerb}
p>\max\{p_1,p_{JL}(\sigma),p_L^+(\sigma)\},
\end{equation}
since $p_L^{+}(\sigma)$ is the largest root of $R(p)$ (when such root exists). Moreover, if $p_1>p_{JL}(\sigma)$, then $R(p_1)>0$ and, by standard properties of second degree polynomials, $p_{JL}(\sigma)<p_L(\sigma)<p_{L}^+(\sigma)$ and $p_1\in(p_L(\sigma),p_L^+(\sigma))$. We have just shown that $p_1$ is never the maximum in \eqref{lowerb}, reaching thus the interval claimed in Theorem \ref{th.sigmalarge}. 

\medskip

As mentioned in the Introduction, due to its technical complication, the detailed comparison between the two critical exponents $p_{JL}(\sigma)$ and $p_L^+(\sigma)$ will be given in the Appendix.

\section{Proof of Theorems \ref{th.sigmasmall} and \ref{th.sigmalarge}}\label{sec.proof}

The results derived so far provide all the ingredients needed to prove Theorems \ref{th.sigmasmall} and \ref{th.sigmalarge}. We recall that some parts of the argument are inspired by the proof of \cite[Theorem 1.1]{IMS26} and, at some points, we will only sketch the argument or quote the adequate reference.

\begin{proof}[Proof of Theorems \ref{th.sigmasmall} and \ref{th.sigmalarge}] Assume $p>\max\{p_{JL}(\sigma),p_L^{+}(\sigma)\}$ or $p_L(\sigma)>p>p_{JL}(\sigma)$ and that $m$, $p$, $\sigma$ belong to the ranges stated in Theorems \ref{th.sigmasmall} and \ref{th.sigmalarge}. We have shown in Section \ref{sec.calc} that, in all these cases, any solution $y(s)$ to the linearized equation \eqref{ylinear} satisfying $y(s)\to K>0$ as $s\to\infty$ has at least three zeros.
	
The proof relies on a backward shooting argument in the phase space associated with the system \eqref{PSsyst}, inspired by the approach of \cite[Theorem 12.2]{GV97} and adapted to the present framework following \cite[Section 3]{IMS26}. As explained in Section \ref{sec.prep}, we consider the family of trajectories $\mathcal C_k$, with $k\in(0,Z(P_2))$, contained in the center manifold of the critical point $Q_1$. We recall that $Z(P_2)$ is given by \eqref{ZP2-XP3} and that the family $\{\mathcal C_k\}_{k\in(0,Z(P_2))}$ is introduced in Remark \ref{rem.Ck}. A key observation is that the number of zeros of the function $y(s)$ solution to \eqref{ylinear} coincides with the number of intersections between the graph of a solution $g(s)$ to \eqref{ODE.g} lying in a suitable neighborhood of $C(\sigma)$ and the constant solution $C(\sigma)$. Equivalently, this quantity gives a lower bound of the number of intersections between a profile $f(\xi)$ solving \eqref{ODE.backward} and the stationary solution \eqref{estacionaria}. We also recall that the phase-space variables $(X,Z)$ are in one-to-one correspondence with the variables $(f,\xi)$ through the transformation \eqref{PSchange}. More precisely, solving \eqref{PSchange} for $f$ and $\xi$, we obtain
\begin{equation*}
	f=\left(\dfrac{m}{\alpha}X\right)^{-\alpha}(mZ)^{2/L},
	\qquad
	\xi=
	\sqrt{\left(\dfrac{m}{\alpha}X\right)^{1-\alpha(m-1)}}
	\,(mZ)^{(m-1)/L}.
\end{equation*}
Consequently, an intersection between a solution of \eqref{ODE.backward} and the stationary profile \eqref{estacionaria} corresponds to a point in phase space having the same coordinates $(X,Z)$ as the equilibrium point $P_2$. Since $Z(P_2)$ is fixed whereas $X$ may take arbitrary positive values, such intersections are represented in the phase space of \eqref{PSsyst} by crossings of the plane $\{Z=Z(P_2)\}$, where $Z(P_2)$ is given by \eqref{ZP2-XP3}. Therefore, the number of intersections of a solution $g(s)$ to \eqref{ODE.g} with the constant $g\equiv C(\sigma)$ coincides with the number of times the corresponding trajectory crosses the plane $\{Z=Z(P_2)\}$. To determine the local behavior of the flow across this plane, we compute the scalar product of the vector field associated with \eqref{PSsyst} and the normal vector $(0,0,1)$. This yields
\begin{equation*}
	H(Y)=Z(P_2)\bigl(\sigma+2+(p-m)Y\bigr),
	\qquad Z(P_2)>0.
\end{equation*}
Hence, the flow crosses the plane $\{Z=Z(P_2)\}$ in the direction of increasing $Z$ whenever $Y>-(\sigma+2)/(p-m)$, and in the opposite direction whenever $Y<-(\sigma+2)/(p-m)$. Recalling the analysis of the linearized equation \eqref{ylinear}, we deduce that whenever the condition \eqref{condAB} holds, there exists $\delta>0$ such that every trajectory $\mathcal C_k$ with $k\in\bigl(Z(P_2)-\delta,Z(P_2)+\delta\bigr)$ crosses the plane $\{Z=Z(P_2)\}$ at least three times. Equivalently, the corresponding solutions of \eqref{ODE.g} intersect the constant equilibrium $g\equiv C(\sigma)$ at least three times. We next introduce the following three sets:
\begin{equation*}
	\begin{split}
		&\mathcal{U}:=\{k\in(0,Z(P_2)): \mathcal{C}_k \ {\rm crosses} \ \{Z=Z(P_2)\}\ {\rm exactly \ twice \ and \ starts \ from} \ Q_2\},\\[2mm]
		&\mathcal{W}:=\{k\in(0,Z(P_2)): \mathcal{C}_k \ {\rm crosses} \ \{Z=Z(P_2)\} \ {\rm at \ least \ three \ times}\},\\[2mm]
		&\mathcal{V}:=(0,Z(P_2))\backslash\left(\mathcal{U}\cup\mathcal{W}\right).
	\end{split}
\end{equation*}	
We first observe that no trajectory $\mathcal{C}_k$ with $k\neq Z(P_2)$ can be tangent to the plane $\{Z=Z(P_2)\}$. Indeed, a tangency may only occur at a point where the flux across this plane vanishes, that is, when $Y=-(\sigma+2)(p-m)$ together with the condition $Z=Z(P_2)$. But these two equalities characterize precisely the trajectory \eqref{estacionariaLine}, which corresponds to $\mathcal{C}_{Z(P_2)}$. Since these points are not  equilibrium points of the system, the uniqueness of solutions to the system \eqref{PSsyst} implies that no other trajectory (in particular no trajectory $\mathcal{C}_k$ with a different value of $k$) can pass through it. It follows immediately from the continuous dependence on the parameter $k$ and from the stability of $Q_2$ established in Lemma \ref{lem.Q23} that both sets $\mathcal{U}$ and $\mathcal{W}$ are open. 

Moreover, Proposition \ref{prop.X0} shows that the unique trajectory entering $P_1$ on its stable manifold originates from $Q_2$, while Proposition \ref{prop.Z0} ensures that the unique trajectory entering $Q_1$ and contained in the plane $\{Z=0\}$ (that is, the trajectory $\mathcal{C}_0$ in the notation of Remark \ref{rem.Ck}) starts from $P_1$. Therefore, a standard continuity argument, together with the local behavior near a (hyperbolic) saddle point (see \cite[Theorem 2.9]{Shilnikov}), yield the existence of $\vartheta>0$ such that every trajectory $\mathcal{C}_k$ with $k\in(0,\vartheta)$ starts from $Q_2$ and remains in a tubular neighborhood of the union of trajectories $Q_2-P_1$ and $P_1-Q_1$. Consequently, for every $k\in(0,\vartheta)$, the trajectory $\mathcal{C}_k$ intersects the plane $\{Z=Z(P_2)\}$ exactly twice, at points close to the corresponding intersections of the unique orbit $Q_2-P_1$ contained in the plane $\{X=0\}$. We have just proved that $(0,\vartheta)\subseteq\mathcal{U}$.

On the opposite side, since $\mathcal{C}_{Z(P_2)}$ coincides with the line \eqref{estacionariaLine}, belonging to the plane $\{Z=Z(P_2)\}$, it follows from the analysis performed in Section \ref{sec.calc} that there exists $\varepsilon\in(0,Z(P_2))$ such that, for any $k\in(Z(P_2)-\varepsilon,Z(P_2))$, the trajectory $\mathcal{C}_{k}$ remains in a suitable tubular neighborhood of the line \eqref{estacionariaLine} and thus intersects at least three times the plane $\{Z=Z(P_2)\}$ (corresponding to the three zeros of the function $y$ solution to \eqref{ylinear}). That is, $(Z(P_2)-\varepsilon,Z(P_2))\subseteq\mathcal{W}$.
	
A standard connectedness argument then shows that $\mathcal V\neq\emptyset$. Therefore, there exists at least one parameter $k_1\in\mathcal V$. By definition of $\mathcal V$, the corresponding trajectory $\mathcal C_{k_1}$ cannot cross the plane $\{Z=Z(P_2)\}$ more than twice, for otherwise $k_1\in\mathcal W$. Moreover, if it crosses this plane exactly twice, then its $\alpha$-limit set cannot be $Q_2$, since otherwise $k_1\in\mathcal U$. We also observe that $\mathcal C_{k_1}$ cannot originate from the critical point $P_2$ in view of Lemma \ref{lem.P2}. Furthermore, $\mathcal C_{k_1}$ cannot be tangent to the plane $\{Z=Z(P_2)\}$, as previously discussed. We next show that $\mathcal C_{k_1}$ crosses the plane $\{Z=Z(P_2)\}$ exactly twice. Assume first by contradiction that $\mathcal C_{k_1}$ never intersects this plane, so that it remains entirely in the region $\{Z<Z(P_2)\}$. Since $k_1<Z(P_2)$, the expansion \eqref{centerQ1} together with the inverse transformation \eqref{change2} yields
\begin{equation*}
	Y=-\dfrac{\sigma+2}{p-m}+\dfrac{\sigma+2}{(p-m)X}\bigl(Z-Z(P_2)\bigr)+o\left(\dfrac1{X^2}\right)<-\dfrac{\sigma+2}{p-m},
\end{equation*}
showing that $\mathcal C_{k_1}$ enters $Q_1$ from the half-space $\{Y<-(\sigma+2)/(p-m)\}$. On the other hand, the flow across the plane $Y=-(\sigma+2)/(p-m)$ points towards the region $\{Y>-(\sigma+2)/(p-m)\}$ when $Z<Z(P_2)$ and towards $\{Y<-(\sigma+2)/(p-m)\}$ when $Z>Z(P_2)$. Since the trajectory remains in $\{Z<Z(P_2)\}$, we infer that it stays entirely in the region $\{Y<-(\sigma+2)/(p-m)\}$. In this region, $Z$ is strictly decreasing whereas $X$ is strictly increasing along the flow. The monotonicity of both $X$ and $Z$ easily excludes oscillatory behavior and implies that the $\alpha$-limit set of $\mathcal C_{k_1}$ must reduce to a critical point. The only critical point contained in the region $\{Y<-(\sigma+2)/(p-m)\}$ is $P_1$, but Lemma \ref{lem.P0P1P3} shows that the trajectories originating from $P_1$ are fully contained in the invariant plane $\{Z=0\}$, yielding a contradiction. A completely analogous argument excludes the possibility that $\mathcal C_{k_1}$ intersects the plane $\{Z=Z(P_2)\}$ exactly once. Indeed, in that case the trajectory would eventually remain in the region $\{Z>Z(P_2)\}$, whose only critical point is $Q_4$. This contradicts Lemma \ref{lem.Q4}, which excludes $Q_4$ as a possible $\alpha$-limit set of $\mathcal C_{k_1}$.

We are thus left with the only remaining possibility that the trajectory $\mathcal C_{k_1}$ crosses the plane $\{Z=Z(P_2)\}$ exactly twice. Taking into account the direction of the flow across the plane $\{Y=-(\sigma+2)/(p-m)\}$, we deduce that the $\alpha$-limit set of $\mathcal{C}_{k_1}$ is contained in the region
\begin{equation*}
	\mathcal D:=
	\left\{
	(X,Y,Z)\in\mathbb R^3:
	X\geq0,\ 
	Y>-\dfrac{\sigma+2}{p-m},\
	0\leq Z<Z(P_2)
	\right\},
\end{equation*}
and is different from the critical point $Q_2$. Moreover, throughout the region $\mathcal D$, both $X$ and $Z$ are strictly increasing with respect to $\eta$. Therefore, an oscillation argument analogous to that employed in the proof of \cite[Lemma 2.9]{IMS26}, based on the assumption of the existence of infinite sequences of maxima and minima, excludes the existence of nontrivial limit sets and shows that the $\alpha$-limit of $\mathcal C_{k_1}$ must reduce to a critical point. Since $Q_2$ has already been excluded, we conclude that the trajectory $\mathcal C_{k_1}$ is a heteroclinic connection from one of the critical points $P_0$, $P_3$ or $Q_5$ to $Q_1$. The corresponding profile satisfies \eqref{beh.inf}, completing the proof.
\end{proof}
Some numerical experiments plotting trajectories reaching $Q_1$ in different ranges of parameters are given in the next Section \ref{sec.num}.

\medskip

\noindent \textbf{Remark.} We are not able to distinguish the starting point of the trajectories between the critical points $P_0$, $P_3$ and $Q_5$. However, any of the local behaviors corresponding to these points gives rise to a peaking self-similar solution. Indeed, on the one hand, the profiles of these solutions satisfy either $f(0)>0$ and $f'(0)=0$ or $f(\xi_0)=0$, $(f^m)'(\xi_0)=0$ for some $\xi_0\geq0$, which is the typical contact condition at the edge of the support for solutions to the porous medium equation. On the other hand, as shown in Theorem \ref{th.peaking}, only the behavior of the profile as $\xi\to\infty$ is relevant for the peaking property (and $k\in(0,Z(P_2))$ implies that the trajectory $\mathcal{C}_k$ corresponds to a profile satisfying \eqref{beh.inf}), while the local behavior near the origin does not make any difference in the limit. Thus, such solutions are peaking self-similar solutions independently of the initial point of the corresponding trajectory. 

\section{Numerical experiments}\label{sec.num}

We gather in this section an outcome of several numerical experiments confirming (and emphasizing) the results of Theorems \ref{th.pSpJL}, \ref{th.sigmasmall} and \ref{th.sigmalarge}. Moreover, the reader can consider the following plots as a visual presentation of the shootings employed in the technical proofs given in Sections \ref{sec.pSpJL} and \ref{sec.proof}. We split the experiments into three ranges.

\medskip 

\noindent \textbf{A. The range $p_S(\sigma)<p<p_{JL}(\sigma)$}. We give in Figure \ref{fig1} the result of a numerical experiment pointing out the shooting from $P_0$ on its unstable manifold. In all the three subfigures, a zoom on the local behavior near the plane $\{Y=-(\sigma+2)/(p-m)\}$ is given, in order to show the very thin line between the behavior of a bunch of trajectories and the behavior of another bunch presenting locally one more crossing of the plane. This is exactly how the proof of Theorem \ref{th.pSpJL} is constructed, see Section \ref{sec.pSpJL}. Moreover, let us remark that the second plot suggest that, between the two peaking solutions, a trajectory entering $Q_1$ from the region $\{Y>-(\sigma+2)/(p-m)\}$ (and thus a self-similar solution presenting the tail behavior \eqref{beh.inf} with $K>C(\sigma)$) is expected to exist. The experiments are done with $N=12$, $\sigma=0.1$, $m=2$, $p=3$, noting that $p_S(\sigma)=2.84$ and $p_{JL}(\sigma)\approx9.724$.

\begin{figure}[ht!]
	\begin{center}
		\subfigure[One and three crossings]{\includegraphics[width=7.5cm,height=6.7cm]{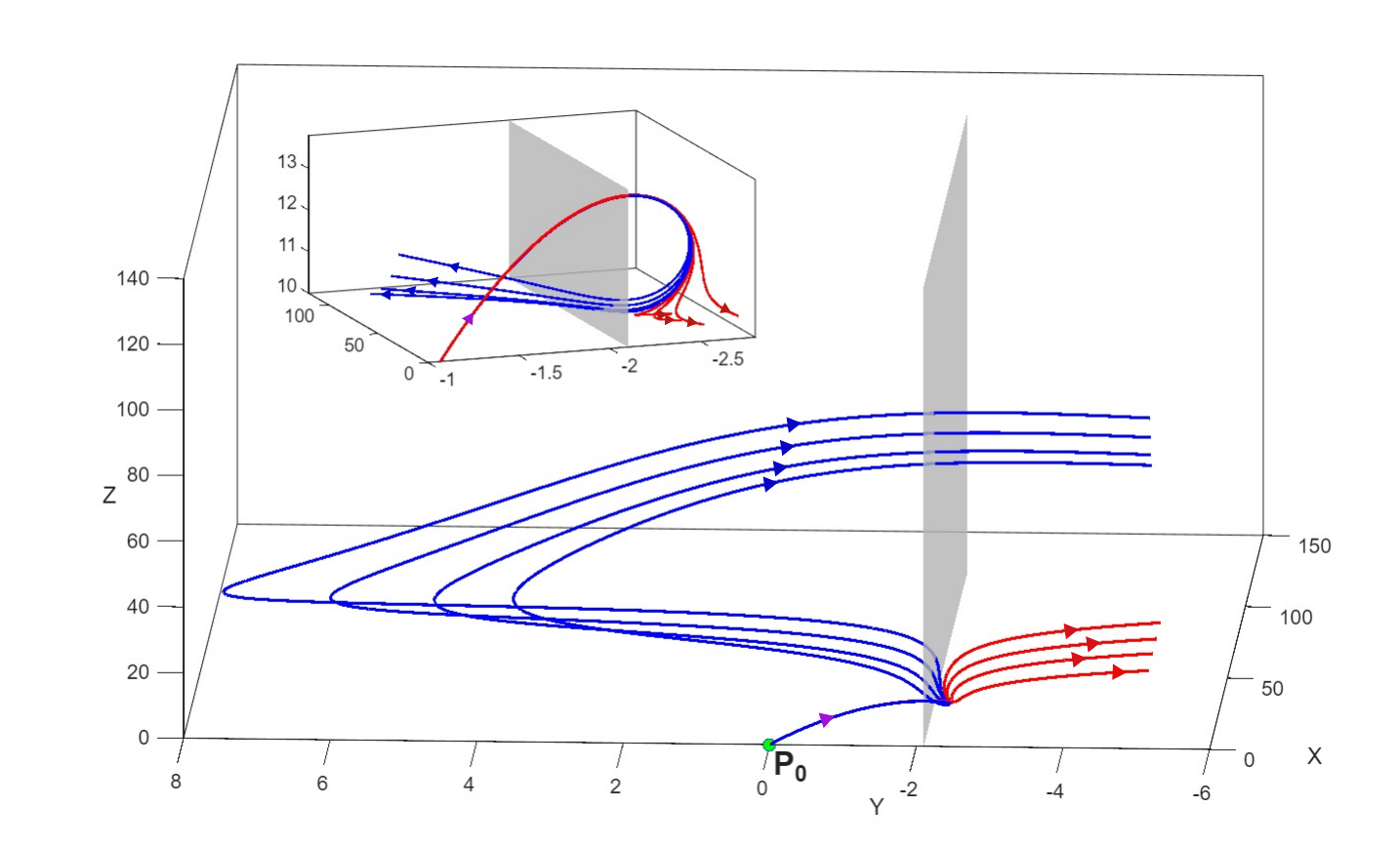}}
		\subfigure[Three and three crossings]{\includegraphics[width=7.5cm,height=6.7cm]{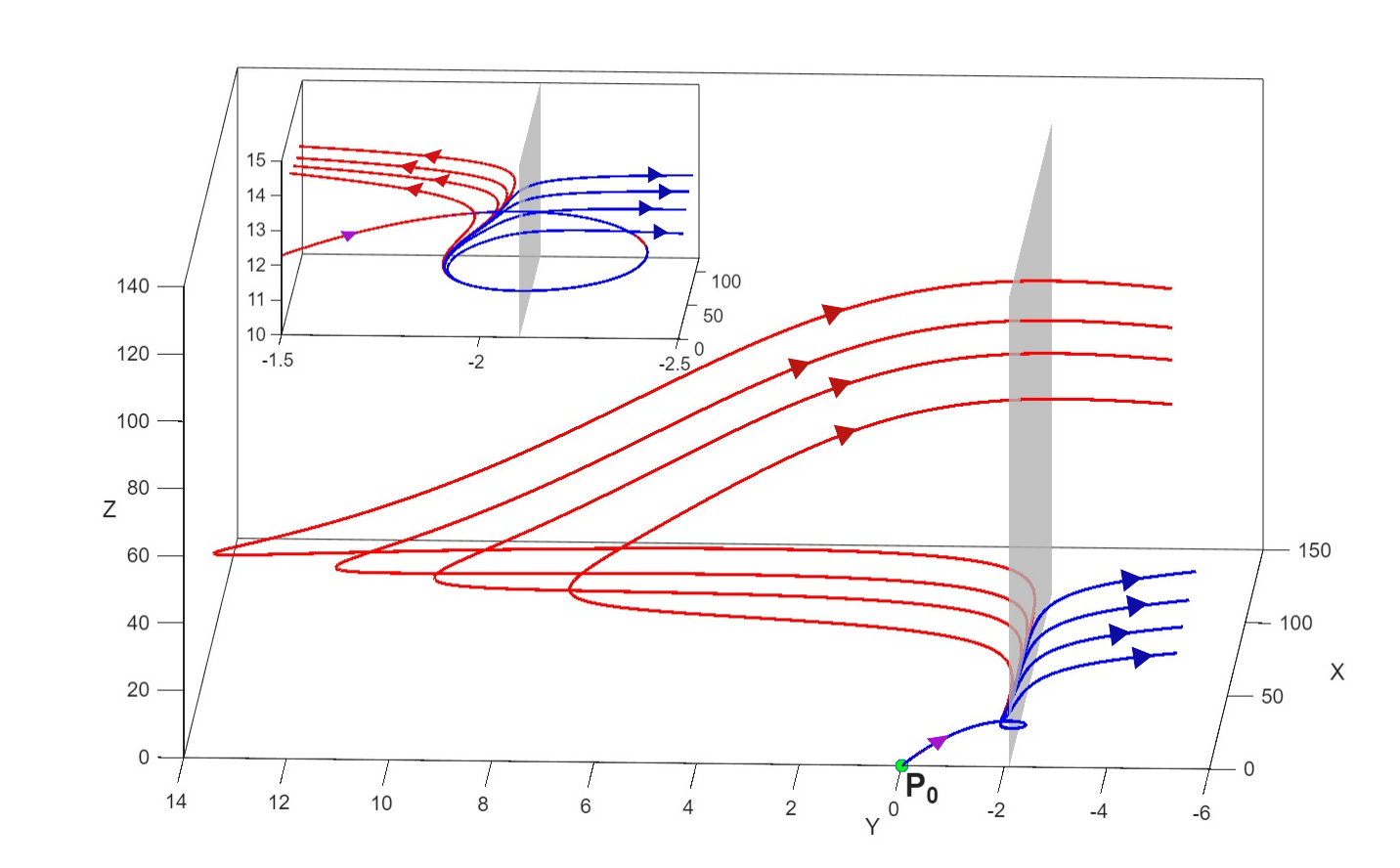}}
		\subfigure[Three and five crossings]{\includegraphics[width=7.5cm,height=6.7cm]{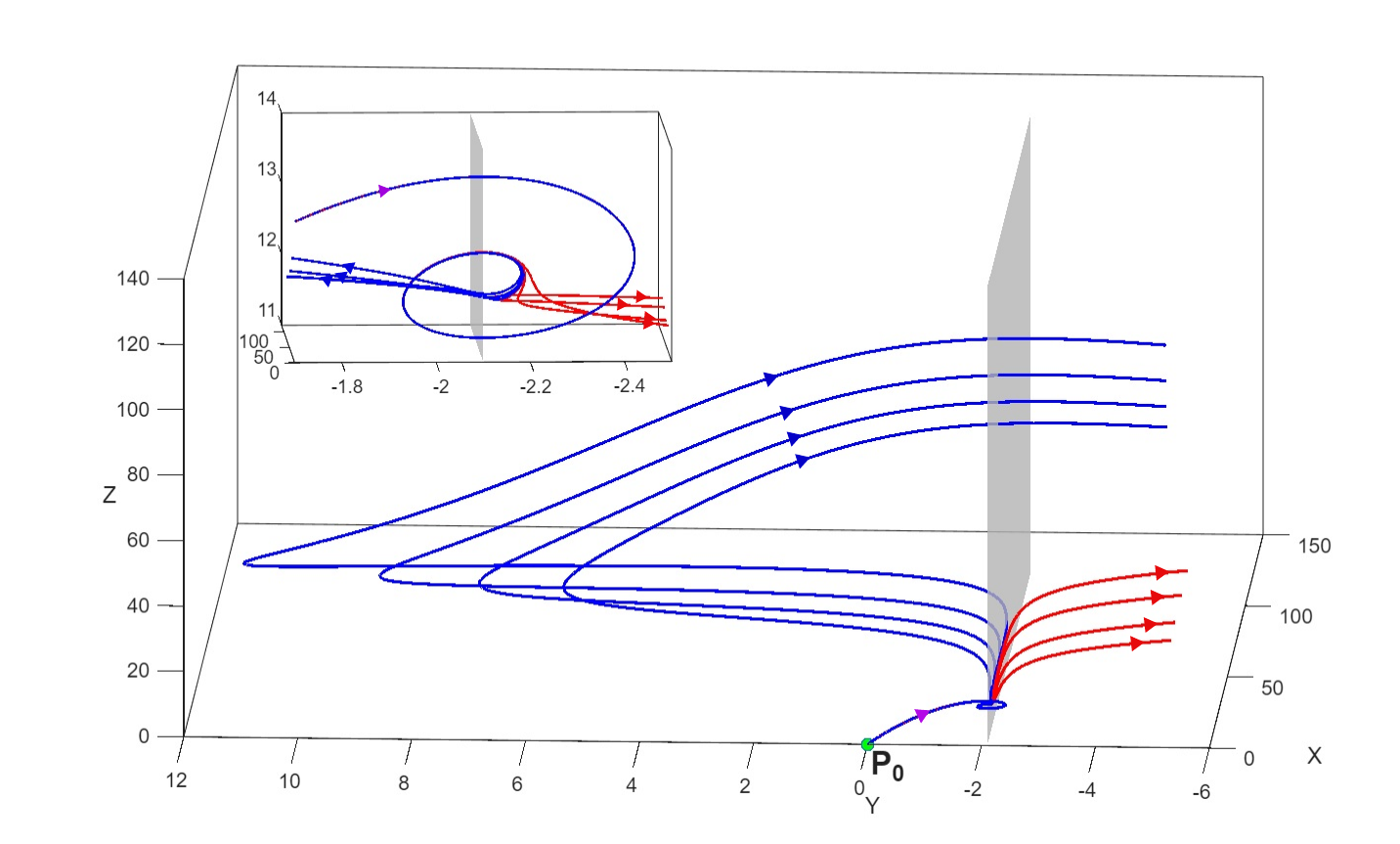}}
	\end{center}
	\caption{Orbits from $P_0$ with different number of oscillations, for $p\in(p_S(\sigma),p_{JL}(\sigma))$}\label{fig1}
\end{figure} 

\medskip 

\noindent \textbf{B. The range $\sigma<\sigma_2(N)$}. In this range we plot in Figure \ref{fig2} the outcome of two experiments, in both cases taking $\sigma=0$ and $N=20$. The experiments are made by shooting backward from $Q_1$, exactly as in the proof of Theorem \ref{th.sigmasmall} given in Section \ref{sec.proof}. For the first plot we took $m=1.1\in(1,m_2)$, since $m_2\approx1.22$, and $p=1.71\in(p_{JL}(0),p_L(0))$, since $p_{JL}(0)\approx1.7042$ and $p_{L}(0)\approx1.7158$. We note a difference in the number of oscillations in the first subfigure, suggesting the existence of a peaking solution, as stated in Theorem \ref{th.sigmasmall}. For the second plot we took $m=2>m_2$ and $p=3.15$ (which belongs to the interval $(p_{JL}(0),p_L(0))\approxeq(3.0985,3.2)$) and all the trajectories behave exactly in the same way, suggesting that all them connect $Q_1$ to $Q_2$ and no peaking solutions exist.

\begin{figure}[ht!]
	\begin{center}
		\subfigure[The peaking range $m\in(1,m_2)$ and $p<p_L(\sigma)$]{\includegraphics[width=7.5cm,height=6.4cm]{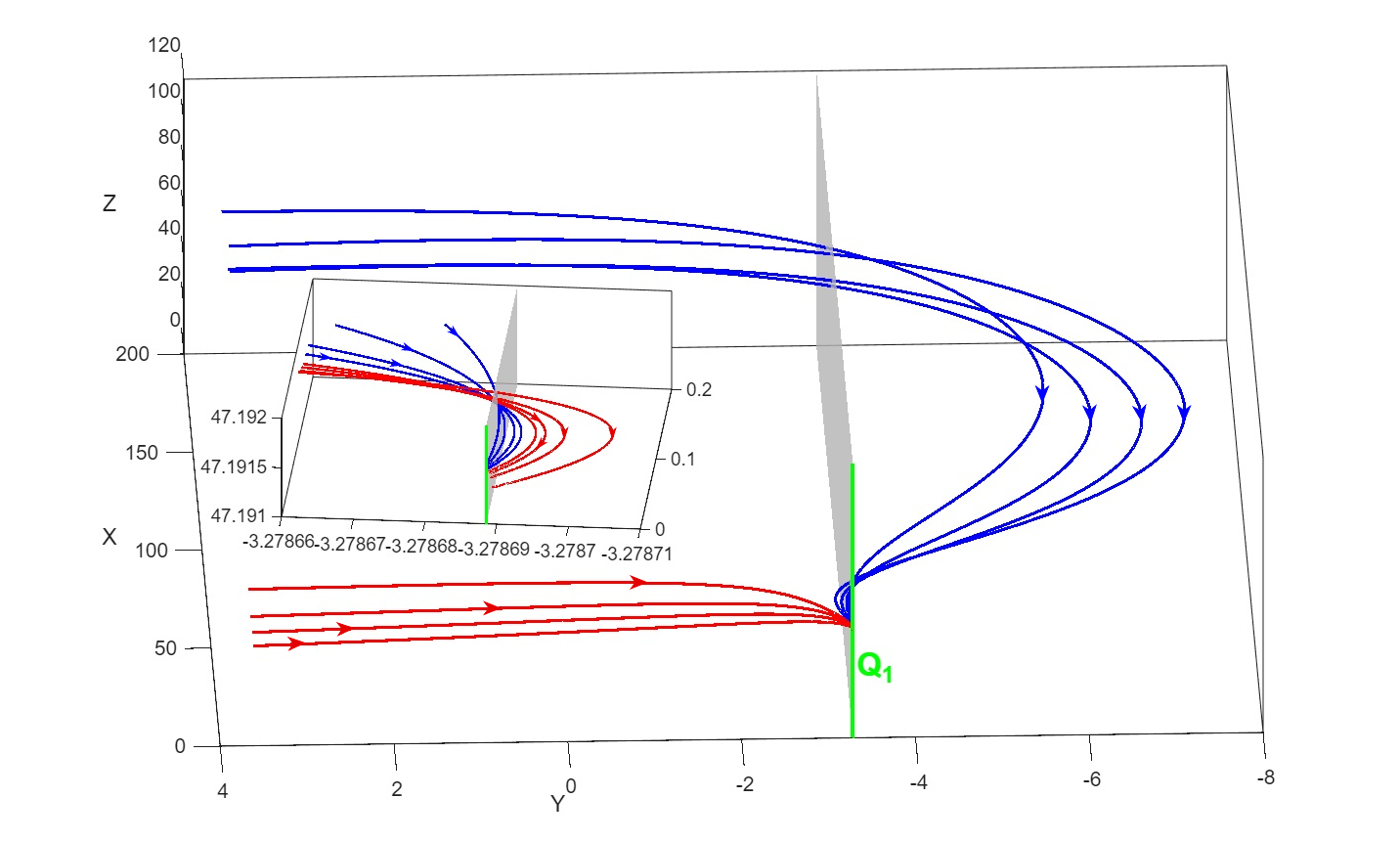}}
		\subfigure[The ``non-peaking" range $m>m_2$]{\includegraphics[width=7.5cm,height=6.4cm]{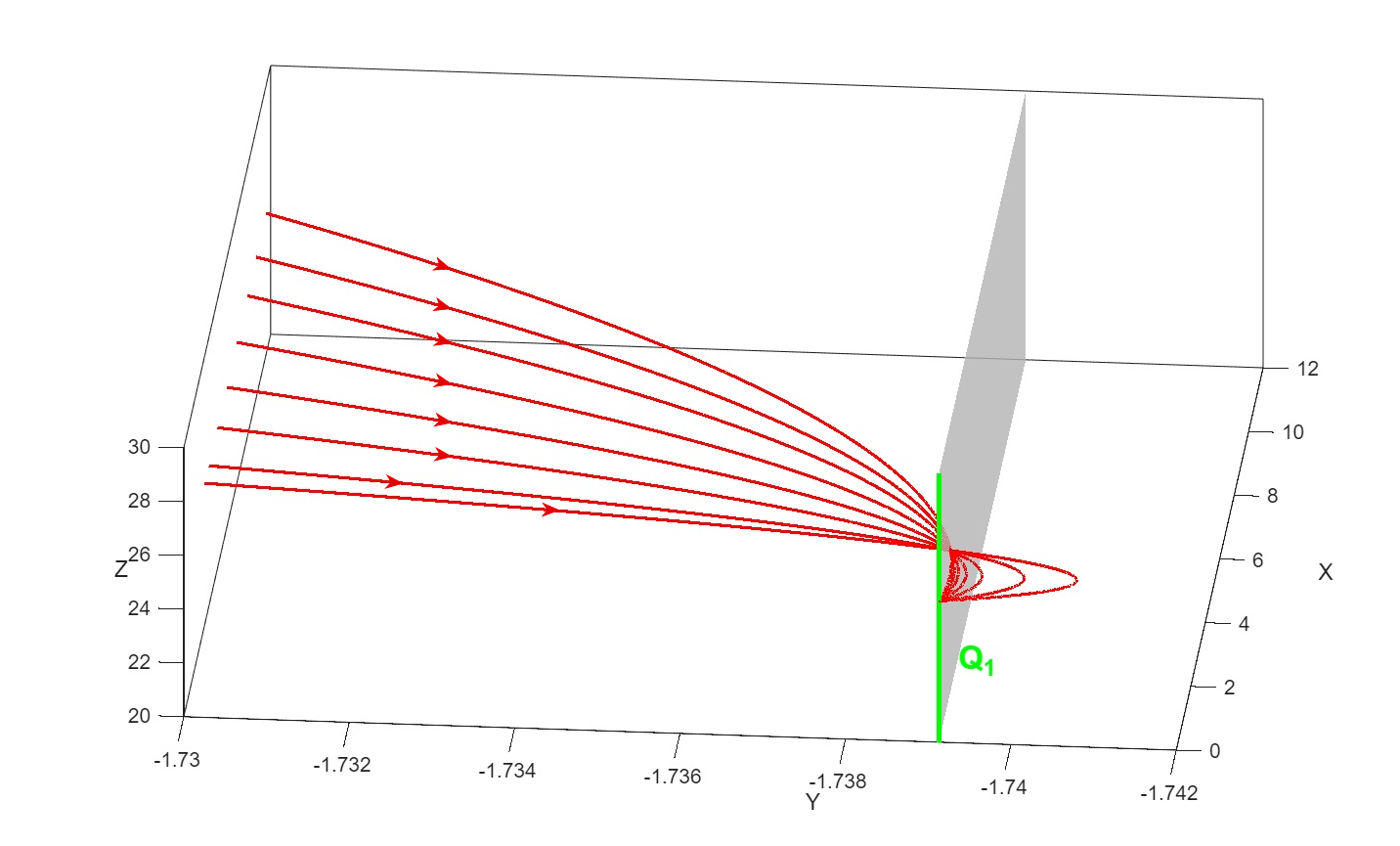}}
	\end{center}
	\caption{Experiments suggesting existence and non-existence of peaking solutions for $\sigma<\sigma_2(N)$}\label{fig2}
\end{figure}  

\medskip 

\noindent \textbf{C. The range $\sigma\in(\sigma_2(N),\sigma_1(N))$}. According to the statement of Theorem \ref{th.sigmalarge}, the discussion after its statement and the proof in the Appendix, this range is split into two subranges according to the relative position of $\sigma$ with respect to the critical value $\sigma_{12}(N)$ defined in \eqref{s12}. We consider first the range $\sigma\in(\sigma_2(N),\sigma_{12}(N))$ and provide in Figure \ref{fig3} a numerical evidence, as well as a visual understanding, of the probably most unexpected phenomenon related to this range, the probable (up to a proof of the non-existence) occurence of a ``gap" in the existence of peaking solutions for $p\in(p_L(\sigma),p_L^+(\sigma))$, as explained in the Remark following Corollary \ref{cor.always}. We observe the different behavior of the orbits in the first and third subplots, but no difference at all in the second one. For these experiments we let $N=20$, $m=3$, $\sigma=1.3$, noting that $\sigma>\sigma^*(20)=1.292$ but $\sigma<\sigma_{12}(20)=1.529$. Taking into account that, with these data, $p_L(\sigma)\approx10.10$, $p_L^+(\sigma)\approx21.034$, we have taken for the three experiments $p=9.5$, $p=15$ and $p=25$ respectively.

\begin{figure}[ht!]
	\begin{center}
		\subfigure[Existence of peaking, $p\in(p_{JL}(\sigma),p_L(\sigma))$]{\includegraphics[width=7.5cm,height=5.5cm]{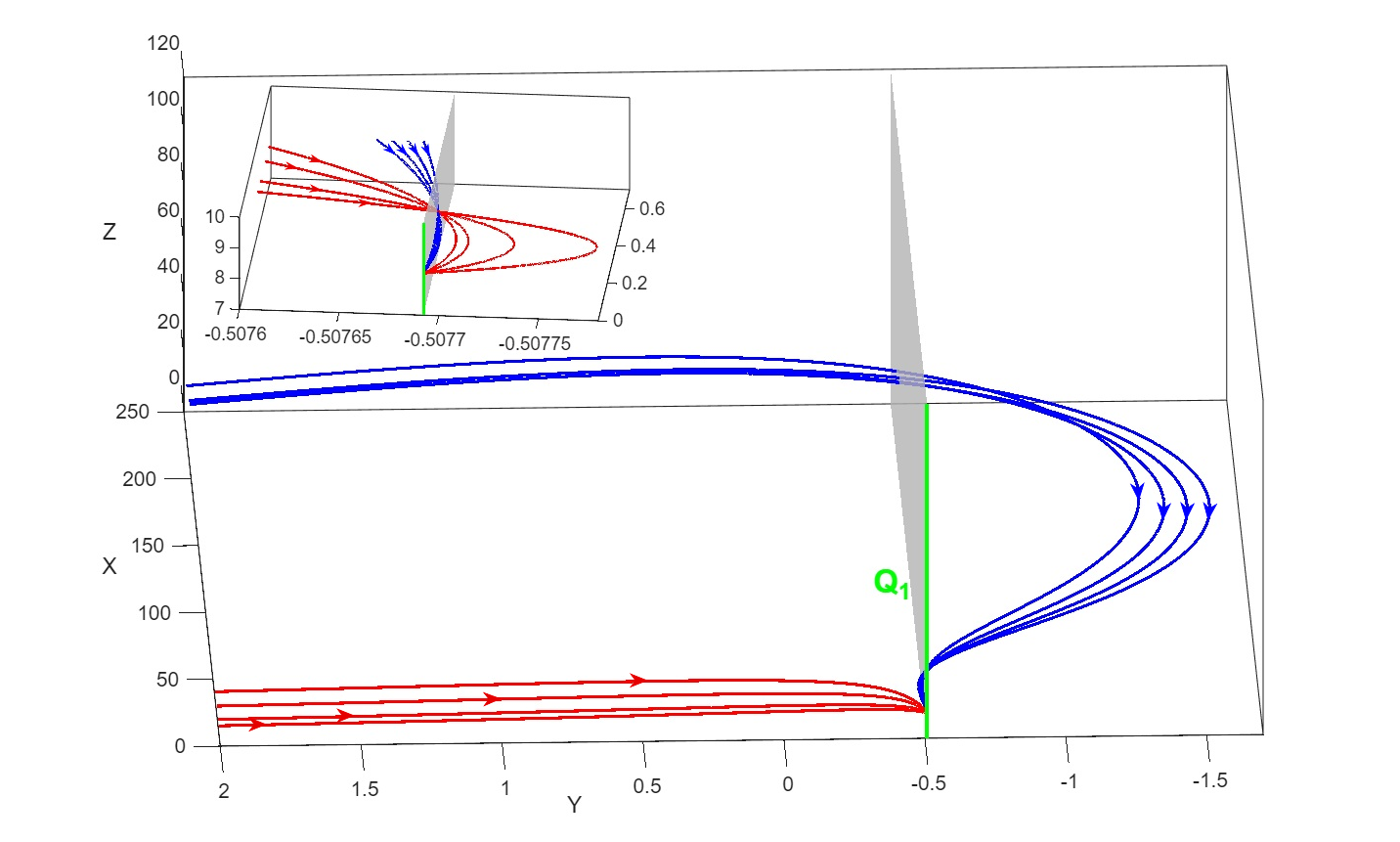}}
		\subfigure[The ``non-peaking" range $p\in(p_L(\sigma),p_L^+(\sigma))$]{\includegraphics[width=7.5cm,height=5.5cm]{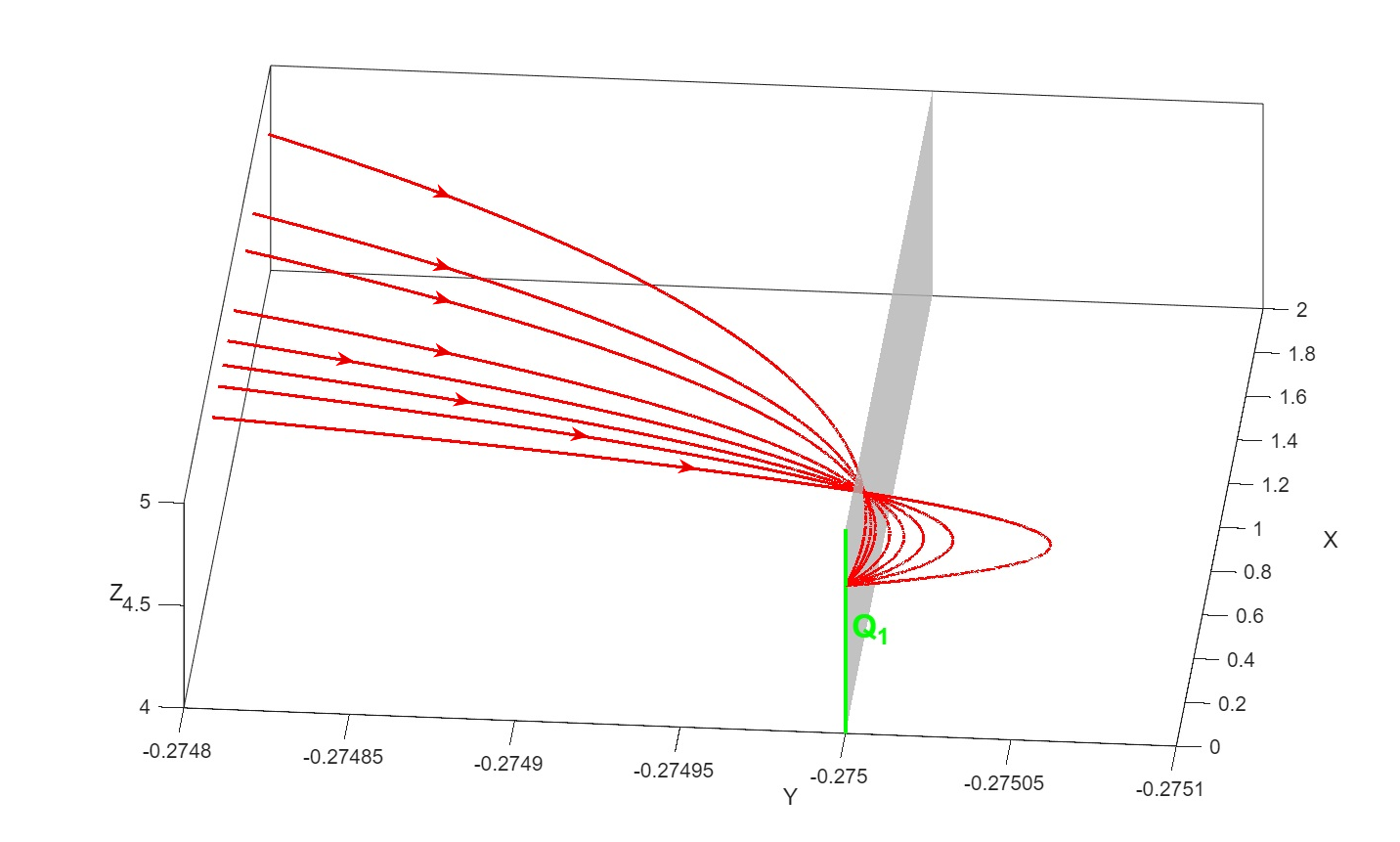}}
		\subfigure[Existence of peaking, $p>p_L^+(\sigma)$]{\includegraphics[width=7.5cm,height=5.5cm]{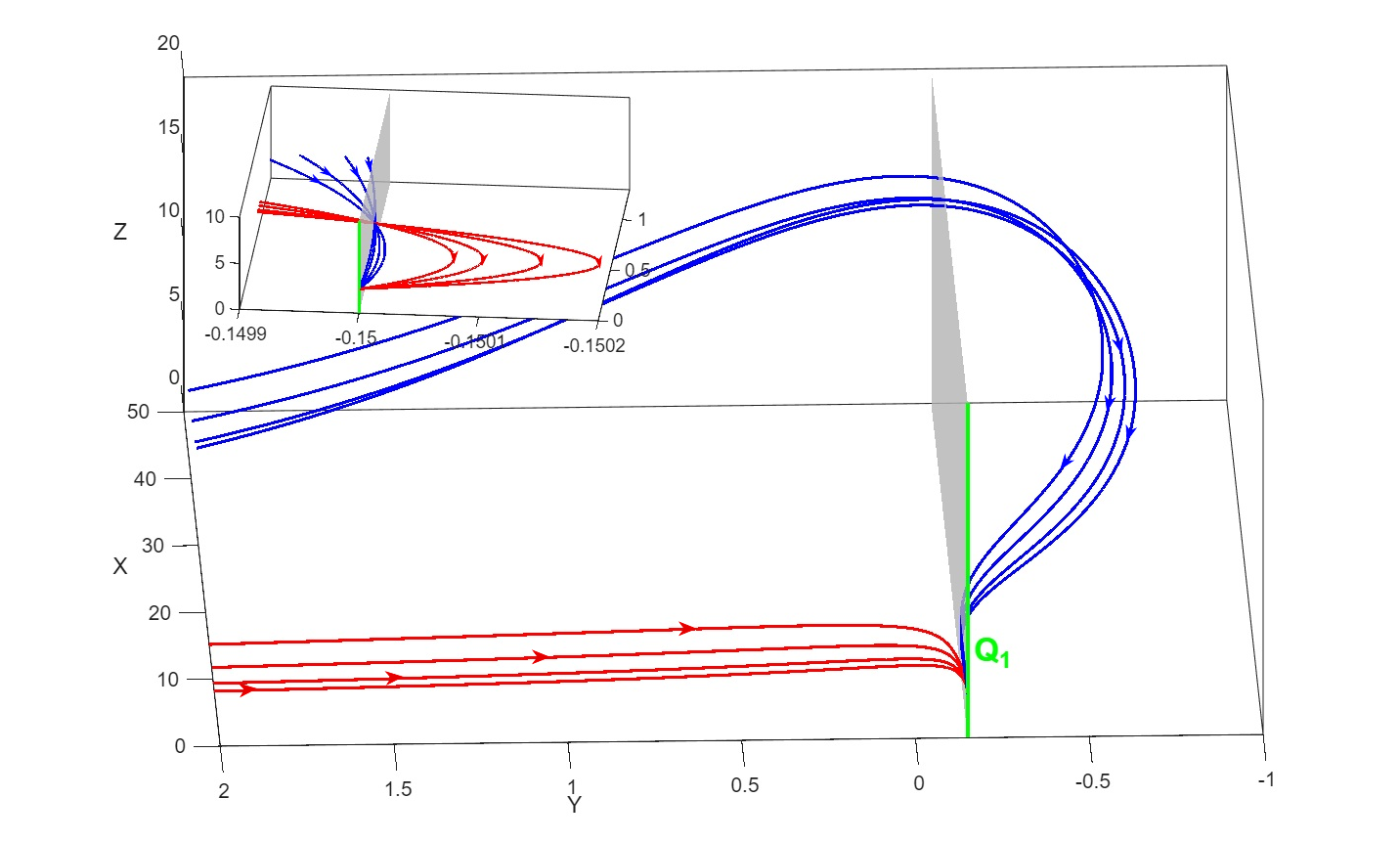}}
	\end{center}
	\caption{Trajectories reaching $Q_1$ for $\sigma\in(\sigma_2(N),\sigma_{12}(N))$}\label{fig3}
\end{figure}

Finally, we exemplify in Figure \ref{fig4} the existence of peaking solutions in the range $\sigma\in(\sigma_{12}(N),\sigma_1(N))$, with $N=20$, $m=3$, $\sigma=2$ and $p=15$, noting that $\sigma_{12}(N)\approx1.529$ and $\sigma_1(N)=2.5$.

\begin{figure}[ht!]
	\begin{center}
		\includegraphics[width=7.5cm,height=5.5cm]{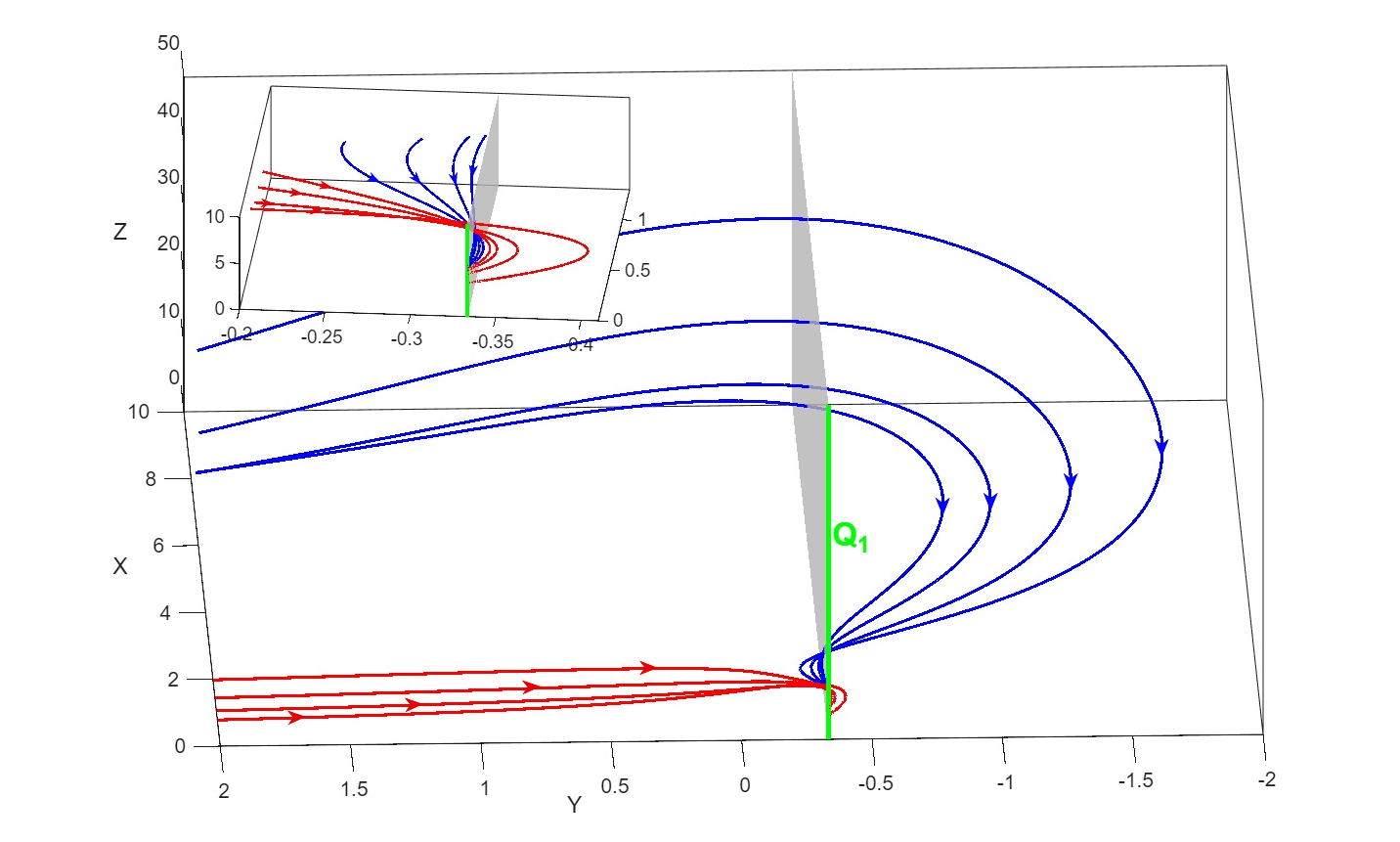}
	\end{center}
	\caption{Existence of peaking solutions for $p\geq p_{JL}(\sigma)$ for $\sigma>\sigma_{12}(N)$.}\label{fig4}
\end{figure}

\section*{Appendix. Completing the proof of Theorem \ref{th.sigmalarge}}
\appendix
\renewcommand{\theequation}{A.\arabic{equation}}

In this appendix, we complete the analysis performed in Section \ref{sec.calc} by proving the two specialized items of Theorem \ref{th.sigmalarge}. Let us thus assume $\sigma_2(N)\leq\sigma<\sigma_1(N)$ and recall the polynomial $R$ introduced in \eqref{comp0} and its rescaled version $\overline{R}$ defined in \eqref{polR}. We split the analysis into several steps. 

\medskip

\noindent \textbf{Step 1. Dimensions $N\in\{11,12\}$.} For $\sigma\geq\sigma_2(N)$, the dominating coefficient of $\overline{R}$ is negative or zero (being zero if and only if $\sigma=\sigma_2(N)$), hence its graph is a concave parabola. We also know that $R(p_{JL}(\sigma))<0$ and $R(p_1)>0$ if $p_1>p_{JL}(\sigma)$, by the definition of $R(p)$ in \eqref{comp0}. It thus follows (by contradiction) that, if $R(p)<0$ for any $p>p_{JL}(\sigma)$, then automatically $p_1<p_{JL}(\sigma)$. It is obvious that the first and third monomials of $R(p)$ are non-positive and we are left to analyze the coefficient of $P$ in the expression of $\overline{R}(P)$, that is, 
$$
S(m,N,\sigma):=(2N+\sigma-26)m-2(N-10).
$$
Let $\sigma_0(N):=-2(N-13)$ be the zero of the coefficient of $m$ in $S(m,N,\sigma)$. We have
$$
\sigma_1(N)-\sigma_0(N)=\frac{9N-104}{4}\leq0,
$$
provided $N\leq 38/3$, which includes the integer dimensions $N\in\{11,12\}$. Thus, if $10<N\leq38/3$, both terms of $S(m,N,\sigma)$ are negative for $\sigma\in(\sigma_2(N),\sigma_1(N))$ and thus 
$$
S(m,N,\sigma)<0, \quad \sigma_2(N)\leq \sigma<\sigma_1(N), \quad 10<N\leq\frac{38}{3}.
$$
Consequently, $\overline{R}(P)<0$ and $R(p)<0$. We thus infer that $p_1<p_{JL}(\sigma)$ and the root $p_L(\sigma)$ does not exist, completing the proof.

\medskip

\noindent \textbf{Step 2. Some new critical exponents.} Assume from now on that $N\geq13$ or, if taking $N$ as a real parameter in \eqref{ODE.backward}, that $N>38/3$. We observe that, whenever $R(p)<0$ for any $p>p_{JL}(\sigma)$, the condition \eqref{condAB} is satisfied for any $p>p_{JL}(\sigma)$. Taking into account that $R(p)$ (or its rescaled simplified version $\overline{R}(P)$) is a second degree polynomial with negative dominating coefficient, this property is ensured when the discriminant of this polynomial is negative. A direct calculation gives that this discriminant is 
\begin{equation}\label{a1}
	\begin{split}
		&\Delta(m,\sigma,N):=D_1(\sigma,N)(m-1)^2+D_2(\sigma,N)(m-1)+(\sigma-6)^2,\\
		&D_1(\sigma,N):=4N^2-(28\sigma+40)N+(\sigma+6)^2,\\
		&D_2(\sigma,N):=8(N-1)-(12N+32)\sigma+2\sigma^2.
	\end{split}
\end{equation}
Let us observe that $D_1(\sigma,N)$ has two roots in terms of $\sigma$, denoted by 
\begin{equation}\label{a2}
	\sigma_{12}(N):=14N-6-8\sqrt{3N^2-2N}, \quad \sigma_{12}^+(N):=14N-6+8\sqrt{3N^2-2N},
\end{equation}
the first of them being already introduced in \eqref{s12}. Taking into account that $N\geq13$, we find that
$$
D_1(\sigma_2(N),N)=\frac{4(N^2-14N+16)^2}{(N-2)^2}>0, \quad D_1(\sigma_1(N),N)=-\frac{47}{16}N^2+\frac{127}{4}N+\frac{49}{4}<0,
$$
hence, by elementary properties of parabolas, 
$$
\sigma_2(N)<\sigma_{12}(N)<\sigma_1(N)<\sigma_{12}^+(N).
$$
Recalling the exponent $\sigma^*(N)$ defined in \eqref{sigmastar}, another direct calculation leads to 
$$
D_1(\sigma^*(N),N)=\frac{64(N+2)^2(N-1)^2}{9(5N-2)^2}>0,
$$
which proves that $\sigma^*(N)<\sigma_{12}(N)$ for any $N>10$. The exponent $\sigma_{12}(N)$ will come decisively into the detailed analysis of the intervals of existence of peaking solutions, as we shall see in the next steps. 

At this point, we introduce the roots of the polynomial $\Delta(m,\sigma,N)$ defined in \eqref{a1} with respect to $m$, which are well defined for $\sigma\neq\sigma_{12}(N)$:
\begin{equation}\label{a3}
	\begin{split}
	&M_{1,2}:=1+\frac{-\sigma^2+(6N+16)\sigma-4(N-1)\pm\sqrt{Q(\sigma;N)}}{D_1(\sigma,N)},\\
	&Q(\sigma;N):=(\sigma+2)(2N-2+\sigma)[(N-2)\sigma-2(N-10)],
	\end{split}
\end{equation}
where $M_2$ corresponds to the plus sign. Since we are dealing with the range $\sigma>\sigma_2(N)$, we remark that $Q(\sigma;N)>0$, hence $M_{1,2}$ are well defined. The analysis performed in this step readily implies that 

$\bullet$ if $\sigma\in(\sigma_2(N),\sigma_{12}(N))$, then $D_1(\sigma,N)>0$ and then $\Delta(m,\sigma,N)>0$ for $m\in(1,M_1)\cup(M_2,\infty)$ and $\Delta(m,\sigma,N)\leq0$ for $m\in[M_1,M_2]$. 

$\bullet$ if $\sigma\in(\sigma_{12}(N),\sigma_1(N))$ then $D_1(\sigma,N)<0$ and then $\Delta(m,\sigma,N)>0$ for $m\in(1,M_1)$ and $\Delta(m,\sigma,N)\leq0$ for $m\geq M_1$.

We thus conclude (at this point, to be extended in the next steps) that there are peaking solutions for any $p>p_{JL}(\sigma)$ in the following two cases: 

(a) $\sigma\in(\sigma_2(N),\sigma_{12}(N))$ and $m\in[M_1,M_2]$, or

(b) $\sigma\in(\sigma_{12}(N),\sigma_1(N))$ and $m\in[M_1,\infty)$.

\medskip 

\noindent \textbf{Step 3. Comparison between $p_{JL}(\sigma)$ and the maximum value of the polynomial $R(p)$.} We are left to study the ranges when $\Delta(m,\sigma,N)>0$, thus we assume this condition in the forthcoming steps. By differentiating $R(p)$ with respect to $p$ and equating the result to zero, we deduce that the unique critical point of the second degree polynomial $R$ is given by 
$$
p_{\max}:=\frac{m[\sigma^2+4(N-7)\sigma-12]-2(\sigma+2)(N-10)}{2[(N-2)\sigma-2(N-10)]}.
$$
We then observe that 
\begin{equation}\label{expr.dif}
p_{\max}-p_{JL}(\sigma)=\frac{(\sigma+2)E(m,\sigma,N)}{(N-4\sigma-10)(N-2)[(N-2)\sigma-2(N-10)]},
\end{equation}
with
\begin{equation}\label{expr.interm}
\begin{split}
E(m,\sigma,N)&:=\left[(-2[(N-2)\sigma-2(N-10)]\sqrt{(\sigma+2)(\sigma+2N-2)}\right.\\
&\left.+\frac{(2N^2-(11N-14)\sigma-22N+20)(N-10)}{2}\right]m\\
&-(N-2)(N-4\sigma-10)(N-10).
\end{split}
\end{equation}
Since $\sigma\in(\sigma_2(N),\sigma_1(N))$, the sign of the left hand side in \eqref{expr.dif} is the same as the sign of the expression $E(m,\sigma,N)$ defined in \eqref{expr.interm}, since the other factors in \eqref{expr.dif} are positive. Let us note at this point that, when $E(m,\sigma,N)<0$ we have $p_{JL}(\sigma)>p_{\max}$. Recalling that $R(p_{JL}(\sigma))<0$, $R(p_{\max})>0$ (from the assumption $\Delta(m,\sigma,N)>0$) and $R(p_1)>0$ (proved in \eqref{Rp}), we infer that whenever $E(m,\sigma,N)<0$ we have $p_1<p_{JL}(\sigma)$ and thus the exponent $p_1$ does not come into play in these ranges.

We next study the sign of the coefficient of $m$ in the expression $E(m,\sigma,N)$, which we denote by $C_m(\sigma,N)$. We have
$$
C_m(\sigma_1(N),N)=-\frac{3(N-10)^2(N-2)}{4}<0.
$$
Moreover, if we denote by 
$$
\sigma_3(N):=\frac{2(N-1)(N-10)}{11N-14}>0
$$
the unique zero of the second term in the coefficient $C_m(\sigma,N)$, we observe that $C_m(\sigma,N)<0$ for any $\sigma>\sigma_3(N)$. We note, on the one hand, that 
$$
\sigma_3(N)-\sigma^*(N)=\frac{8(N-1)(N-10)(N+2)}{3(11N-14)(5N-2)}>0,
$$
hence $\sigma_3(N)>\sigma^*(N)$. On the other hand, recalling that $\sigma_{12}(N)$ is the smallest root of $D_1(\sigma,N)$, we compute
$$
D_1(\sigma_3(N);N)=-\frac{128(N-1)(N+2)(N^2-14N+16)}{(11N-14)^2}<0,
$$
since $N^2-14N+16>0$ for $N\geq13$. We infer from basic properties of parabolas that $\sigma_3(N)>\sigma_{12}(N)$. The previous analysis thus proves that there is a peaking solution for any $p\geq p_{JL}(\sigma)$, provided $\sigma\in(\sigma_3(N),\sigma_1(N))$.

\medskip

\noindent \textbf{Step 4. Assume next that $\sigma\leq\sigma_3(N)$.} In this case, we can estimate the sign of the coefficient $C_m(\sigma,N)$ which multiplies $m$ in \eqref{expr.interm} by comparing the squares of the two terms of it and thus get rid of the square root. We then have 
\begin{equation*}
	\begin{split}
(2N\sigma-4N-4\sigma+40)^2&(\sigma+2)(\sigma+2N-2)-\frac{(2N^2-(11N-14)\sigma-22N+20)^2(N-10)^2}{4}\\
&=-\frac{(N-2)(N-4\sigma-10)}{4}P_3(\sigma,N),
	\end{split}
\end{equation*}
where 
\begin{equation*}
	\begin{split}
P_3(\sigma,N)&:=4(N-2)\sigma^3+(9N^2-44N+180)\sigma^2-(28N^3-572N^2+2928N-1232)\sigma\\
&+4(N-1)(N-10)(N^2-19N+82).
	\end{split}
\end{equation*}
The polynomial $P_3(\sigma,N)$ is a cubic polynomial with respect to $\sigma$. Moreover, it satisfies 
$$
P_3(\sigma_2(N),N)=\frac{4(N-10)^2(N^2-14N+16)^2}{(N-2)^2}>0,
$$
$$
P_3(\sigma_3(N),N)=-\frac{1024(N-1)(N-10)(N+2)(N^2-14N+16)^2}{(11N-14)^3}<0
$$
and 
$$
P_3(\sigma^*(N),N)=\frac{64(N-1)^2(N-10)(N+2)^2(23N^2-300N+124)}{27(5N-2)^3}>0,
$$
the latter holding true since $N\geq13$. We next compute the expression $P_3(\sigma_{12}(N),N)$. Due to the presence of the square root in the expression of $\sigma_{12}(N)$, the calculation becomes extremely technical. With the help of the Maple symbolic calculation software, one finds the following very tedious expression:
\begin{equation}\label{intermxx}
P_3(\sigma_{12}(N),N)=Q_1(N)\sqrt{3N^2-2N}+Q_2(N),
\end{equation}
with 
\begin{equation*}
\begin{split}
&Q_1(N):=128(-209N^3+596N^2-508N+112), \\
&Q_2(N):=128(362N^4-1153N^3+1204N^2-436N+32).
\end{split}
\end{equation*}
By employing a Taylor development centered at 13 (the minimal value of $N$ allowed at this point), one finds that $Q_1(N)<0$, $Q_2(N)>0$ for any $N\geq13$. Thus, with the aid of the same software and the same strategy of expressing it as a polynomial with respect to $N-13$, one can prove that 
$$
Q_1^2(N)(3N^2-2N)-Q_2^2(N)<0, \quad N\geq13,
$$
and conclude from \eqref{intermxx} that $P_3(\sigma_{12}(N),N)>0$. We deduce that $P_3$ has a root (as polynomial in $\sigma$) $\sigma_{02}(N)\in(\sigma_{12}(N),\sigma_3(N))$. Recalling that $P_3(\sigma,N)$ has been derived starting from the coefficient $C_m(\sigma,N)$, we infer from \eqref{expr.interm} that $E(m,\sigma,N)<0$ still holds true for $\sigma\in(\sigma_{02}(N),\sigma_3(N))$ and thus, proceeding as in Step 3, the existence of peaking solutions is guaranteed for any $p>p_{JL}(\sigma)$, provided $\sigma\in(\sigma_{02}(N),\sigma_1(N))$.

Letting then $\sigma\in(\sigma_2(N),\sigma_{02}(N))$, and denoting by $m_0(\sigma,N)$ the unique root of the expression $E(m,\sigma,N)$ with respect to $m$ (which can be made explicit, but we refrain from giving here its very tedious expression), then peaking solutions exist for $p>p_{JL}(\sigma)$ provided $1<m<m_0(\sigma,N)$. We can gather the outcome of the last two steps as follows: there are peaking solutions for any $p>p_{JL}$ if 

$\bullet$ either $\sigma_{02}(N)\leq\sigma<\sigma_1(N)$ and any $m>1$ 

$\bullet$ or $\sigma_2(N)<\sigma<\sigma_{02}(N)$ and $m\in(1,m_0(N))$.

\medskip 

\noindent \textbf{Step 5. Assume that $\sigma\in(\sigma_{12}(N),\sigma_{02}(N))$.} We conclude from Step 4 that peaking solutions exist for $m\in(1,m_0(\sigma,N))$ and $p\geq p_{JL}(\sigma)$. But we derive from \eqref{expr.dif} and \eqref{expr.interm} that, for $m=m_0(\sigma,N)$, we have $E(m_0(\sigma,N),\sigma,N)=0$ and thus $p_{\max}=p_{JL}(\sigma)$. Recalling \eqref{comp0}, we readily find that $R(p_{JL}(\sigma))<0$, whence $R(p_{\max})<0$. By the definition of $p_{\max}$, we deduce that $R(p)<0$ for any $p$, which in particular implies that the discriminant (defined in \eqref{a1}) $\Delta(m_0(\sigma,N),\sigma;N)$ is negative. It follows that $m_0(\sigma,N)>M_1$ (the only root lying above one if $\sigma>\sigma_{12}(N)$). Hence
$$
(1,m_0(\sigma,N))\cup [M_1,\infty)=(1,\infty),
$$
and the previous analysis, together with the result of Step 1, allows to conclude that there exist peaking solutions for any 
$$
m>1, \quad \sigma\in(\sigma_{12}(N),\sigma_1(N)), \quad p\geq p_{JL}(\sigma),
$$
completing thus the proof of the first item of Theorem \ref{th.sigmalarge}.

\medskip 

\noindent \textbf{Step 6. Assume that $\sigma_2(N)<\sigma<\sigma_{12}(N)$.} The same argument as in Step 5 allows now to prove that $\Delta(m_0(\sigma,N),\sigma,N)<0$ and thus $m_0(\sigma,N)\in[M_1,M_2]$, where $M_1$ and $M_2$ are the roots defined in \eqref{a3}. Joining this fact with the conclusion of Step 1, we deduce that there exist peaking solutions for 
$$
\sigma\in(\sigma_2(N),\sigma_{12}(N)), \quad m\in(1,M_2], \quad p\geq p_{JL}(\sigma).
$$
We have thus proved the first part of the second item of Theorem \ref{th.sigmalarge}. Next, if $m>M_2$ we are in the situation where $E(m,\sigma,N)>0$ and thus $p_{\max}>p_{JL}(\sigma)$. Since we are under the assumption $\Delta(m,\sigma,N)>0$, it follows that $R(p_{\max})>0>R(p_{JL}(\sigma))$ and we deduce that $p_{JL}(\sigma)$ is smaller than the smallest root of $R(p)$, which is, in this range, $p_L(\sigma)$. By the previous analysis, we find that 
\begin{equation}\label{a5}
R(p)<0, \quad p\in(p_{JL}(\sigma),p_L(\sigma))\cup(p_{L}^{+}(\sigma),\infty).
\end{equation}
Recalling the exponent $p_1$ introduced in \eqref{eq.p1}, we still need to compare $p_1$ to $p_{JL}(\sigma)$ in order to complete the intervals of existence of peaking self-similar solutions. We have shown in \eqref{compp1pJL} that $p_{JL}(\sigma)>p_1$ for $\sigma>\sigma^*(N)$, which includes the whole interval $\sigma\in(\sigma_2(N),\sigma_{12}(N))$ if $N<18$ and the interval $(\sigma^*(N),\sigma_{12}(N))$ when $N\geq18$, as shown in \eqref{compsigma}. We thus have established the existence of peaking solutions for all the interval in the right hand side of \eqref{a5} if 
$$
\max\{\sigma_2(N),\sigma^*(N)\}<\sigma<\sigma_{12}(N),
$$
provided $m>M_2$. 

Finally, if $N\geq18$, we are left with the interval $\sigma\in(\sigma_2(N),\sigma^*(N))$ where $p_1>p_{JL}(\sigma)$ if $m>m_2$, according to the arguments leading to \eqref{compp1pJL}. In this case $R(p_1)>0$, as shown in \eqref{Rp}, which ensures that $p_1\in(p_L(\sigma),p_L^+(\sigma))$. Since for $p\in (p_{JL}(\sigma),p_1)$ the condition \eqref{condAB} is not satisfied, we deduce from \eqref{a5} and this relative position of $p_1$ that the second interval in the right hand side of \eqref{a5} is the interval of $p$ for which the existence of peaking self-similar solutions is guaranteed, completing the proof of the second item of Theorem \ref{th.sigmalarge}. 

\medskip

\noindent \textbf{Remark.} In the same range $N\geq18$ and $\sigma\in(\sigma_2(N),\sigma^*(N))$, one can also show that $M_2<m_2$, where $M_2$ is defined in \eqref{a3} and $m_2$ is given by \eqref{m12}. This implies that, for $m\in[M_2,m_2]$, one has $p_1<p_{JL}(\sigma)<p_L(\sigma)$ and there exist peaking solutions for $p\in[p_{JL}(\sigma),p_L(\sigma))$. We refrain from giving a proof of the fact that $m_2>M_2$ here.

\bigskip

\noindent \textbf{Acknowledgements} The authors are partially supported by the Project PID2024-160967NB-I00 funded by AEI (Spain) and FEDER.

\bigskip

\noindent \textbf{Data availability} Our manuscript has no associated data.

\bigskip

\noindent \textbf{Conflict of interest} The authors declare that there is no conflict of interest.

\bibliographystyle{plain}

\end{document}